\documentclass[11pt,a4paper,twoside]{article}
\usepackage{amsthm, amsfonts,amsmath, color}

\newtheorem{theorem}{Theorem}
\newtheorem{corollary}{Corollary}
\newtheorem{definition}{Definition}
\newtheorem{example}{Example}
\newtheorem{lemma}{Lemma}
\newtheorem{proposition}{Proposition}
\newtheorem{remark}{Remark}

\def\Ric{\operatorname{Ric}}
\def\tr{{\rm\,trace\,}}

\title{$\ast$-$\eta$-Ricci solitons on weak Kenmotsu $f$-manifolds}
\author{Vladimir Rovenski
\footnote{Department of Mathematics, University of Haifa, Mount Carmel, 3498838 Haifa, Israel
\newline e-mail: {\tt vrovenski@univ.haifa.ac.il}
}}

\date{}

\begin{document}

\maketitle

\begin{abstract}
Recent interest among geometers in $f$-structures of K.~Yano is due to the study of topo\-logy and dynamics of contact foliations
and generalized A.\,Weinstein conjectures. Weak metric $f$-structures,
introduced by the author and R.\,Wolak as a generalization of Hermitian structure, as well as $f$-structure
allow for a fresh perspective on the classical theory.
An~important case of such manifolds, which is locally a twisted product, is a weak $\beta f$-Kenmotsu manifold defined as a generalization of K.~Kenmotsu's concept.
In~this paper, the concept of the $\ast$-Ricci tensor of S.\,Tashibana is adapted to weak metric $f$-manifolds, the interaction of $\ast$-$\eta$-Ricci soliton with the weak $\beta f$-Kenmotsu structure is studied and new characteristics of $\eta$-Einstein metrics are~obtained.


\end{abstract}

\section{Introduction}

Ricci solitons ({RS}), i.e., self-similar solutions of the Ricci flow $\partial g/\partial t = -2\Ric_g$,
generalize Einstein metrics $g$ on a smooth manifold $M$, i.e., solutions of $\Ric = \lambda\,g$, where $\lambda$ is a real constant, see~\cite{CLN-2006}.
The study of RS is also motivated by the fact that some compact manifolds don't admit Einstein~metrics.
A~Riemannian~metric $g$, a smooth vector field $V$ and a scalar $\lambda$
represent a RS if
 $\frac12\pounds_V\,g + \Ric = \lambda\,g$,
where $\pounds$ is the Lie derivative.
A RS is {expanding}, {steady}, or {shrinking} if $\lambda$ is negative, zero, or positive, respectively.
If $V$ is zero or a Killing vector field, i.e. $\pounds_V\,g=0$, then RS reduces to Einstein metric and we get a trivial soliton.
Some authors study the generalization of RS to $\eta$-RS defined by
\begin{align}\label{E-sol-eta}
 (1/2)\pounds_V\,g+\Ric=\lambda\,g +\mu\,\eta\otimes\eta,
\end{align}
see \cite{cho2009ricci},
where $\eta$ is a 1-form and $\mu$ is a smooth function on $M$
and tensor product notation $(\eta\otimes\eta)(X,Y)=\eta(X)\,\eta(Y)$ is used.
if $\pounds_V\,g=0$, then we get $\eta$-Einstein metrics
 $\Ric=\lambda\,g +\mu\,\eta\otimes\eta$.

An interesting question is \textit{how Ricci-type solitons interact with various geometric structures (e.g. contact structure or foliation) on a manifold}.
When does an almost contact manifold equipped with a Ricci-type soliton carry camonical (e.g. Einstein-type)~metrics?
Contact Riemannian geometry is of growing interest due to its important role in physics, e.g., \cite{blair2010riemannian}.
S.\,Tanno \cite{Tan-1969} classified connected almost contact metric manifolds with automorphism groups of maximal
dimensions, as follows:
(i)~homogeneous normal contact Riemannian manifolds with constant $f$-holomorphic sectional curvature if the $\xi$-sectional curvature $K(\xi, X) > 0$;
(ii)~Riemannian products if $K(\xi, X)=0$;
(iii)~warped pro\-ducts of the real line and a K\"{a}hler manifold, if $K(\xi, X)<0$.
The concept of a warped product is also important for  mathematical physics and
general relativity: some of Einstein metrics and spacetime models are of this~type.
Z.~Olszak \cite{olszak1991normal} cha\-racterized
a class (iii), known as $\beta$-Kenmotsu manifolds ($\beta$-{KM}) for $\beta=const\ne0$ (and defined by K. Kenmotsu \cite{kenmotsu1972class}, when $\beta=1$),
by the following equality:
\begin{equation*}
 (\nabla_{X} f)Y=\beta\{g(f X, Y)\,\xi -\eta(Y)\,f X\} .
\end{equation*}
If $\beta\ne0$ is a smooth function on $M$, we get locally a twisted product, see \cite{pr-1993}.

The $\ast$-Ricci tensor, defined by S.~Tashibana \cite{Tash-1959} on almost Hermitian manifolds, was then applied by T.~Hamada \cite{Hamada-2002} to
real hypersurfaces in non-flat complex space forms. For almost contact metric manifolds
it is given (using the curvature tensor $R_{{X},{Y}}Z=([\nabla_X, \nabla_Y]-\nabla_{[X,Y]})Z$ by
\begin{align}\label{E-ast-Ricci-tensor}
 \Ric^\ast(X,Y)
 = (1/2)\,\tr\{Z\to R_{X, fY} fZ\}
  \quad (X,Y\in\mathfrak{X}_M) .
\end{align}
If $\Ric^\ast=\lambda\,g +\mu\,\eta\otimes\eta$ for some $\lambda,\mu\in C^\infty(M)$, then we get $\ast$-$\eta$-Einstein metrics.
Further, for $\mu=0$
we get $\ast$-Einstein~metrics $\Ric^\ast=\lambda\,g$.
G.\,Kaimakamis and K.\,Panagiotidou \cite{KP-2014} introduced
$\ast$-RS by
\begin{align}\label{E-ast-Ricci-sol}
 (1/2)\pounds_V\,g + \Ric^\ast = \lambda\,g .
\end{align}
A.\,Ghosh and D.S.\,Patra \cite{Gosh-Patra2018} first studied $\ast$-RS on Sasakian manifolds.
Then the study of $\ast$-RS on almost contact manifolds, particularly, on almost Sasakian, cosymplectic and (para) KM, has gained popularity,
e.g. \cite{Li-2022,Mondal-2024,Patra-2024,ven-2019}.
Inspired by the works on $\ast$-RS, first S.\,Dey and S.\,Roy \cite{DeyRoy-2020} for $V=\xi$,
and then S.~Dey, S.~Sarkar and A.\,Bhattacharyya \cite{DSB-2024} for any $V\in\mathfrak{X}_M$ introduced the
concept of $\ast$-$\eta$-RS on almost contact
manifolds by replacing the Ricci tensor in
\eqref{E-sol-eta} with the $\ast$-Ricci~tensor:
\begin{align}\label{E-sol-ast-eta0}
 (1/2)\pounds_V\,g+\Ric^\ast=\lambda\,g +\mu\,\eta\otimes\eta.
\end{align}
When $\mu=0$, \eqref{E-sol-ast-eta0} reduces to the $\ast$-RS equation \eqref{E-ast-Ricci-sol}.
Some authors, see \cite{DSB-2024,ven-2022}, studied $\ast$-$\eta$-RS on KM $M^{2n+1}(f,\xi,\eta,g)$, and found conditions when such $g$ are $\eta$-Einstein or even Einstein metrics.
If~$V=\nabla v$ in \eqref{E-sol-ast-eta0} for a function $v\in C^\infty(M)$, then we obtain a gradient $\ast$-$\eta$-RS
${\rm Hess}_{\,v} +\Ric^\ast = \lambda\,g +\mu\,\eta\otimes\eta$, where Hess is the Hessian operator.
If $\lambda,\mu\in C^\infty(M)$ in \eqref{E-sol-ast-eta0}, then we get an almost $\ast$-$\eta$-RS.


K.\,Yano's \cite{yano-1961} $f$-structure defined by a (1,1)-tensor field $f$ of rank $2n$ on a smooth manifold $M^{2n+s}$ such that $f^3 + f = 0$,
serves as a higher-dimensional analog
of {almost complex structures} ($s=0$) and {almost contact structures} ($s=1$).
The tangent bundle splits into complementary subbundles:
 $TM=f(TM)\oplus\ker f$,
and the~restriction of $f$ to the $2n$-dimensional distribution $f(TM)$ is a {complex structure}.
A~submanifold $M$ of an almost complex manifold $(\bar M,J)$ that satisfies the condition $\dim(T_xM\cap J(T_xM))=const>0$
naturally possesses an $f$-structure, see~\cite{L-1969}.
An $f$-structure is a~special case of an {almost product structure}, given a pair of complementary orthogonal distributions of a Riemannian mani\-fold $(M,g)$, with Naveira's 36 distinguished classes. Foliations appear when one or both distributions are involutive.
A {metric $f$-structure} \cite{b1970,FP06,FP07,gy-1970,YK-1985} occurs when the subbundle $\ker f$ of $M^{2n+s}$ is parallelizable.
In this scenario, there exist vector fields $\{\xi_i\}_{1\le i\le s}$ spanning $\ker f$
with dual 1-forms $\{\eta^i\}_{1\le i\le s}$ and a
Riemannian metric $g$, e.g., \cite{b1970}, satisfying
\begin{align*}
 {f}^2 = -{\rm id} + \sum\nolimits_{\,i}{\eta^i}\otimes {\xi_i}, \quad
 {\eta^i}({\xi_j})=\delta^i_j,\quad
 g({f}X,{f}Y)=g(X,Y)-\!\sum\nolimits_{\,i}{\eta^i}(X)\,{\eta^i}(Y)\quad (X,Y\in\mathfrak{X}_M) .
\end{align*}
Various broad classes of metric $f$-manifolds are known; for instance, $f$-KM (KM when $s=1$, see \cite{kenmotsu1972class}),
characterized in terms of warped products of $\mathbb{R}^s$ and K\"{a}hler manifolds, see~\cite{SV-2016}.
Recent interest among geometers in $f$-manifolds is also motivated by the study of the topology and dyna\-mics of contact foliations, especially the existence of closed leaves. Contact foliations gene\-ralize to higher dimensions the flow of the Reeb vector field on contact manifolds,
see \cite{Fin-2024}.

\smallskip

In \cite{RWo_2}, we initiated the study of ``weak" $f$-structures on smooth $(2n+s)$-dimensional manifolds,
i.e., the complex structure on the contact distribution
of a metric $f$-structure is replaced with a nonsingular skew-symmetric tensor.
These generalize the $f$-structure (the weak almost contact structure for $s=1$) and its satellites, allow us to take a new look at classical theory and find new applications of Killing vector fields and totally geodesic foliations, Einstein-type metrics and Ricci-type solitons.

In \cite{rst-58}, we defined $\ast$-$\eta$-RS (in particular, $\ast$-RS) of weak almost contact manifolds, studied how they interact with
weak $\beta$-KM, and proved that under certain conditions the soliton metric is Einstein metric.
In this paper, developing the approach \cite{rst-58}, we define the $\ast$-Ricci tensor of a weak metric $f$-manifold (with $s>1$) and study when a weak
$\beta f$-{KM} (see Definition~\ref{D-wK}), equipped with $\ast$-$\eta$-RS (see Definition~\ref{D-ast-eta-RS}), in particular, $\ast$-RS, carries an $\eta$-Einstein metric.

The paper consists of an introduction and four sections.
Sections~\ref{sec:01} and \ref{sec:02-f-beta} contain preliminary material on weak metric $f$-manifolds and their distinguished class
of weak $\beta f$-KM.
Section~\ref{sec:03} introduces the $\ast$-Ricci tensor and $\ast$-scalar curvature
for weak metric $f$-manifolds and establishes their relation with the classical notions (Theorem~\ref{T-1-ast}).
Section~\ref{sec:04} contains our results on the interaction of $\ast$-$\eta$-RS with the $\beta f$-KM.
Proposition~\ref{T-lambda0} derives the sum of the soliton constants under certain conditions.
Our~main results generalize theorems of other authors and show that for a weak $\beta f$-KM equipped with a $\ast$-$\eta$-RS, $g$ is an $\eta$-Einstein metric under certain assumptions: the potential vector field $V$ is a contact vector field (Theorem~\ref{T-002a}); $g$ is an $\eta$-Einstein metric (Theorem~\ref{T-002}); $V$ belongs to $\ker f$ (Theorem~\ref{T-004}); and the soliton is a gradient almost $\ast$-$\eta$-soliton (Theorem~\ref{T-005}).

\section{Weak Metric $f$-Manifolds}
\label{sec:01}

This section contains the basics of a weak metric $f$-structure as a higher dimensional analog of the weak almost contact metric structure,
see \cite{RWo_2,rst-43,rst-57}.


\begin{definition}
\label{D-basic}\rm
A~\textit{weak metric $f$-structure} on a smooth manifold $M^{2n+s}\ (n>0,\,s>1)$ is a set $({f},Q,{\xi_i},{\eta^i},g)$, where
${f}$ is a skew-symmetric $(1,1)$-tensor of rank $2\,n$, $Q$ is a self-adjoint nonsingular $(1,1)$-tensor,
${\xi_i}\ (1\le i\le s)$ are orthonormal
vector fields, ${\eta^i}$ are dual 1-forms, and $g$ is a Riemannian metric on $M$, satisfying
\begin{align}\label{2.1}
 {f}^2 = -Q + \sum\nolimits_{\,i}{\eta^i}\otimes {\xi_i},\quad {\eta^i}({\xi_j})=\delta^i_j,\quad
 Q\,{\xi_i} = {\xi_i}, \\
\label{2.2}
 g({f} X,{f} Y)= g(X,Q\,Y) -\sum\nolimits_{\,i}{\eta^i}(X)\,{\eta^i}(Y)\quad (X,Y\in\mathfrak{X}_M),
\end{align}
and $M^{2n+s}({f},Q,{\xi_i},{\eta^i},g)$ is called a \textit{weak metric $f$-manifold}.
\end{definition}

Putting $Y=\xi_j$ in \eqref{2.2}, and using ${\eta^i}({\xi_j})=\delta^i_j$, we get
\begin{align}\label{2.2-eta}
 \eta^j(X) = g(X,\xi_j);
\end{align}
thus, ${\xi_j}$ is orthogonal to the distribution ${\cal D}=\bigcap_{\,i=1}^s \ker{\eta^i}$.
For a more intuitive under\-standing of the role of $Q$ in the $f$-structure, we explain the following properties:
\[
 {f}\,{\xi_i}=0,\quad {\eta^i}\circ{f}=0,\quad \eta^i\circ Q=\eta^i,\quad [Q,\,{f}]=0 .
\]
By \eqref{2.1}, $f^2\xi_i=0$ is true. From this and \eqref{2.1}, we get
 $f^3 + fQ = 0$.
By this, $Q\xi_i=\xi_i$ and $f^2\xi_i=0$ we get $0=-f^3\xi_i=fQ\xi_i=f\xi_i$.
By $f\xi_i=0$, \eqref{2.2-eta}, and the skew-symmetry of $f$, we get $\eta^i(fX)=g(fX,\xi_i)=-g(X,f\xi_i)=0$.
From this and condition ${\rm rank}\,f=2n$, we conclude that $f$ the distribution ${\cal D}$
of a weak metric $f$-structure is ${f}$-invariant, ${\cal D}=f(TM)$ and $\dim{\cal D}=2\,n$.
By this and $f^3 + fQ = 0$, we get $f^3X=f^2(fX)=-QfX$; hence, $f^3 + Qf = 0$.
This and $f^3 + fQ = 0$ yield $fQ=Qf$. By symmetry of $Q$ and $Q\xi_i=\xi_i$, we get $\eta^i(QX)=g(QX,\xi_i)=g(X,Q\xi_i)=g(X,\xi_i)=\eta^i(X)$.
Therefore, $TM$ splits as complementary orthogonal sum of
${\cal D}$ and~$\ker f$.

A weak metric $f$-structure $({f},Q,{\xi_i},{\eta^i},g)$ is said to be {\it normal} if the following tensor is zero:
\begin{align*}
 {\cal N}^{\,(1)}(X,Y) = [{f},{f}](X,Y) + 2\sum\nolimits_{\,i}d{\eta^i}(X,Y)\,{\xi_i}\quad (X,Y\in\mathfrak{X}_M),
\end{align*}
where the Nijenhuis torsion
of a (1,1)-tensor ${S}$ and the exterior derivative of a 1-form ${\omega}$ are given~by
\begin{align*}
\notag
 & [{S},{S}](X,Y) = {S}^2 [X,Y] + [{S} X, {S} Y] - {S}[{S} X,Y] - {S}[X,{S} Y]\quad (X,Y\in\mathfrak{X}_M), \\
 & d\omega(X,Y) = (1/2)\,\{X({\omega}(Y)) - Y({\omega}(X)) - {\omega}([X,Y])\}\quad (X,Y\in\mathfrak{X}_M).
\end{align*}
Using the Levi-Civita connection $\nabla$ of $g$, one can rewrite $[S,S]$ as
\begin{align}\label{4.NN}
 [{S},{S}](X,Y) = ({S}\nabla_Y{S} - \nabla_{{S} Y}{S}) X - ({S}\nabla_X{S} - \nabla_{{S} X}{S}) Y .
\end{align}
Recall the following formulas (with $X,Y\in\mathfrak{X}_M$):
\begin{align}
\label{3.3B}
 (\pounds_{X}\,\eta)(Y) & = X(\eta(Y)) - \eta([X, Y]) ,\\
 \label{3.3C}
 (\pounds_{Z}\,g)(X,Y) & = g(\nabla_X Z, Y) + g(X, \nabla_Y Z).
\end{align}


A~distribution ${\cal D}^\bot\subset TM$ (integrable or not) is said to be {totally geodesic} if
its second fundamental form vanishes: $\nabla_X Y+\nabla_Y X\in{\cal D}^\bot$ for any vector fields $X,Y\in{\cal D}^\bot$ --
this is the case when {any geodesic of $M$ that is tangent to ${\cal D}^\bot$ at one point is tangent to ${\cal D}^\bot$ at all its points}.

\begin{proposition}[see \cite{rst-43}]
The condition ${\cal N}^{\,(1)}=0$ for a weak metric $f$-structure implies
$\pounds_{{\xi_i}}{f} = d{\eta^j}({\xi_i}, \cdot)
= 0$,
$d{\eta^i}({f} X,Y) - d{\eta^i}({f} Y,X)
= (1/2)\,\eta^i([\widetilde QX, fY])$ and
\begin{align}\label{Eq-normal-2}
 \nabla_{\xi_i}\,\xi_j\in{\cal D},\quad  [X,\xi_i]\in{\cal D}\quad (1\le i,j\le s,\ X\in{\cal D}).
\end{align}
Moreover, $\nabla_{\xi_i}\,\xi_j+\nabla_{\xi_j}\,\xi_i=0$, that is ${\cal D}^\bot$ is a totally geodesic distribution.
\end{proposition}

The {fundamental $2$-form} $\Phi$ on $M^{2n+s}({f},Q,\xi_i,\eta^i,g)$ is defined by
 $\Phi(X,Y)=g(X,{f} Y)$ for all $X,Y\in\mathfrak{X}_M$.
Recall the co-boundary formula for exterior derivative $d$ on a $2$-form $\Phi$,
\begin{align*}
 3\,d\Phi(X,Y,Z) =
 X\Phi(Y,Z) + Y\Phi(Z,X) + Z\Phi(X,Y) -\Phi([X,Y],Z) - \Phi([Z,X],Y) - \Phi([Y,Z],X) .
\end{align*}


Let $\Ric^\sharp$ be a~(1,1)-tensor adjoint to the {Ricci tensor}
-- the suitable trace of the curvature tensor:
\begin{equation*}
 {\rm Ric}\,(X,Y) = {\rm trace}_{\,g}(Z\to R_{\,Z,X}\,Y),\quad
 R_{X, Y} = [\nabla_{X},\nabla_{Y}] - \nabla_{[X,Y]}\quad (X,Y,Z\in\mathfrak{X}_M).
\end{equation*}
The scalar curvature of a Riemannian manifold $(M,g)$ is defined as ${r}={\rm trace}_{\,g}\Ric={\rm trace}\Ric^\sharp$.


\begin{definition}\rm
A weak metric $f$-manifold
(see Definition~\ref{D-basic})  is said to be \textit{$\eta$-Einstein},~if
\begin{align}\label{Eq-2.10}
 \Ric = a\,g - a\sum\nolimits_{\,i}\eta^i\otimes\eta^i + (a+b)\,\bar\eta\otimes\bar\eta
 \quad \mbox{for some}\ \ a,b\in C^\infty(M).
\end{align}
\end{definition}

The formula $\Ric = a\,g + b\sum\nolimits_{\,i}\eta^i\otimes\eta^i + (a+b)\sum\nolimits_{\,i\ne j}\eta^i\otimes\eta^j$
is equivalent to the definition \eqref{Eq-2.10}.
Taking the trace of \eqref{Eq-2.10}, gives the scalar curvature $r = (2\,n + s)\,a + s\,b$.
For $s=1$ the definition \eqref{Eq-2.10} reduces to an $\eta$-Einstein metric on a weak almost contact manifold.

\section{Weak $\beta f$-Kenmotsu Manifolds}
\label{sec:02-f-beta}

Here, we review fundamental properties of weak $\beta f$-KM (Theorem~\ref{T-2.0}), give their geometrical interpretation using the twisted product (Theorem~\ref{T-2.1} and Example~\ref{Ex-sHK}), and show that a weak $\beta f$-KM with $\bar\xi$-parallel Ricci tensor is an $\eta$-Einstein manifold (Theorem~\ref{Th-01}).
The following definition generalizes the notions of $\beta$-KM ($s=1$), $f$-KM ($\beta=1$), see \cite{FP06,FP07,ghosh2019ricci,SV-2016}, and weak $\beta$-KM ($s=1$), see~\cite{rst-58}.

\begin{definition}[see~\cite{rst-57}]\label{D-wK}\rm
A weak metric $f$-manifold $M^{2n+s}({f},Q,\xi_i,\eta^i,g)$ will be called a \textit{weak $\beta f$-KM} (a~\textit{weak $f$-KM}) when $\beta\equiv1)$,
if
\begin{align}\label{2.3-f-beta}
 (\nabla_{X}\,{f})Y=\beta\{g({f} X, Y)\,\bar\xi -\bar\eta(Y){f} X\}\quad (X,Y\in\mathfrak{X}_M),
\end{align}
where
$\bar\xi=\sum\nolimits_{\,i}\xi_i$, $\bar\eta=\sum\nolimits_{\,i}\eta^i$, and $\beta\in C^\infty(M)$.
If $\beta\equiv0$, then \eqref{2.3-f-beta} defines a \textit{weak ${\mathcal C}$-manifold}.
\end{definition}

Note that $\bar\eta(\xi_i)=\eta^i(\bar\xi)=1$ and $\bar\eta(\bar\xi)=s$.
Taking $X=\xi_j$ in \eqref{2.3-f-beta} and using ${f}\,\xi_j=0$, we get $\nabla_{\xi_j}{f}=0$, which implies
$\nabla_{\xi_i}\,\xi_j\perp {\cal D}$. This and the 1st equality in \eqref{Eq-normal-2} give
\begin{align}\label{Eq-normal-3}
 \nabla_{\xi_i}\,\xi_j =0\quad  (1\le i,j\le s),
\end{align}
thus, ${\cal D}^\bot$ of a weak $\beta f$-KM is tangent to a foliation with flat totally geodesic~leaves.

\begin{lemma}\label{Lem-1}
For a weak $\beta f$-KM 
the following formulas are true:
\begin{align}\label{2.3b}
 & \nabla_{X}\,\xi_i = \beta\{X -\sum\nolimits_{\,j}\eta^j(X)\,\xi_j\}\quad (1\le i\le s,\ X\in\mathfrak{X}_M),\\
 \label{2.3c}
 & (\nabla_{X}\,\eta^i)(Y) = \beta\{g(X,Y) -\sum\nolimits_{\,j}\eta^j(X)\,\eta^j(Y)\}\quad (1\le i\le s,\ X,Y\in\mathfrak{X}_M) ,\\
 \label{E-nS-10b}
 & (\nabla_X\,Q)Y
 = -\beta\big\{\bar\eta(Y)\,\widetilde Q X + g(\widetilde Q X, Y)\,\bar\xi\big\}  \quad (X,Y\in\mathfrak{X}_M).
\end{align}
In particular,
$\nabla_{\xi_i}Q=0$ and $\tr Q=const$.
\end{lemma}

\begin{proof}
Taking $Y=\xi_i$ in \eqref{2.3-f-beta} and using $g({f} X, \xi_i)=0$ and $\bar\eta(\xi_i)=1$, yields ${f}(\nabla_X\,\xi_i -\beta X)=0$.
Since ${f}$ is non-degenerate on ${\cal D}$ and has rank $2\,n$, we get $\nabla_X\,\xi_i -\beta X = \sum\nolimits_{\,p}c^p\xi_p$.
The inner~product with $\xi_j$ gives $g(\nabla_X\,\xi_i,\xi_j) = \beta\,g(X,\xi_j) - c^j$.
Using \eqref{Eq-normal-2} and \eqref{Eq-normal-3}, we find $g(\nabla_X\,\xi_i,\xi_j)=g(\nabla_{\xi_i}X,\xi_j)=0$; hence, $c^j=\beta\,\eta^j(X)$.
This proves~\eqref{2.3b}. Using $(\nabla_{X}\,\eta^i)(Y) = g(\nabla_{X}\,\xi_i, Y)$, see \eqref{3.3B}, and \eqref{2.3b}, we get~\eqref{2.3c}.
Next, using \eqref{2.1}, we find
\begin{align}\label{E-nS-10b1}
 (\nabla_X\,Q)Y = \nabla_X\,(QY)-Q\nabla_X\,Y = -(\nabla_X\,f^2)Y
 + \sum\nolimits_{\,i}\big\{\eta^i(Y)\nabla_X\,\xi_i + (\nabla_X\,\eta^i)(Y)\,\xi_i\big\}\,\bar\xi
\end{align}
for all  $X,Y\in\mathfrak{X}_M$.
Then using \eqref{2.3-f-beta}, \eqref{2.3b} and the equality $Q={\rm id}+\widetilde Q$ in \eqref{E-nS-10b1}, we find
\begin{align*}
\notag
 (\nabla_X\,Q)Y &= -f(\nabla_X\,f)Y -(\nabla_X\,f)fY \\
\notag
 & + \beta\,\bar\eta(Y)\{X-\sum\nolimits_{\,j}\eta^j(X)\,\xi_j\} + \beta\{g(X,Y)-\sum\nolimits_{\,j}\eta^j(X)\,\eta^j(Y)\} \\
 & = -\beta\,\bar\eta(Y)\,\widetilde Q X -\beta\,g(\widetilde Q X, Y)\,\bar\xi \quad (X,Y\in\mathfrak{X}_M),
\end{align*}
that completes the proof of \eqref{E-nS-10b}. Taking the trace of \eqref{E-nS-10b} gives $X(\tr Q)=0$ for all $X\in\mathfrak{X}_M$.
\end{proof}

Using the first formula of \eqref{2.3b},
for a weak $\beta f$-KM we obtain
\begin{equation}\label{2.3d}
 \pounds_{\xi_i}\,g = 2\,\beta\{g - \sum\nolimits_{\,j}\eta^j\otimes\eta^j\} .
\end{equation}

\begin{theorem}[see \cite{rst-57}]\label{T-2.0}
For a weak $\beta f$-KM
the following conditions are valid:
\begin{align*}
 {\cal N}^{\,(1)} = 0,\quad
 d\eta^i = 0,\quad
 d\Phi = 2\,\beta\,\bar\eta\wedge\Phi.
\end{align*}
\end{theorem}

\begin{definition}[\cite{rst-58}]
\rm
An even-dimensional Riemannian manifold $(\bar M, \bar g)$ equipped with a  non-degene\-rate skew-symmetric {\rm (1,1)}-tensor $J$
(other than a complex structure) is called a \textit{weak Hermitian manifold} if
the symmetric tensor $J^{\,2}$ is negative definite.
Furthermore, if $\bar\nabla J=0$, where $\bar\nabla$ is the Levi-Civita connection of $\bar g$, then
$(M,\bar g, J)$ is called a \textit{weak K\"{a}hler~manifold}.
\end{definition}

\begin{remark}\rm
L.\,P. Eisenhart \cite{E-1923} proved that if a
Riemannian manifold admits a parallel symmetric 2-tensor other than the constant multiple of metric tensor, then it is reducible.
Then several authors studi\-ed the (skew-)symmetric parallel 2-tensors and classified them, for example, \cite{Gupta-2020,H-2022}.
\end{remark}

Let $(\bar M,\bar g)$ be a Riemannian manifold.
A \textit{twisted product} $\mathbb{R}^s\times_\sigma\bar M$
is the product $M=\mathbb{R}^s\times\bar M$ with the metric $g=dt^2\oplus \sigma^2\,\bar g$,
where $\sigma>0$ is a smooth function on $M$.
Set $\xi_i=\partial_{\,t_i}$.
The~Levi-Civita connections, $\nabla$ of $g$ and $\bar\nabla$ of $\bar g$, are related as follows:

\noindent\ \
(i) $\nabla_{\xi_i}\,\xi_j= 0$,
$\nabla_X\,\xi_i=\nabla_{\xi_i}X = \xi_i(\log\sigma)X$ for $X\perp Span\{\xi_1,\ldots,\xi_s\}$.

\noindent\ \
(ii) $\pi_{1*}(\nabla_XY) = -g(X,Y)\,\pi_{1*}(\nabla\log\sigma)$,
where $\pi_1: M \to\mathbb{R}^s$ is the orthoprojector.

\noindent\ \
(iii) $\pi_{2*}(\nabla_XY)$ is the lift of $\bar\nabla_XY$, where $\pi_2: M \to\bar M$ is the orthoprojector.

\noindent
If $\sigma(t_1,\ldots,t_s)>0$ is a smooth function on $\mathbb{R}^s$, then we get a \textit{warped product} $\mathbb{R}^s\times_\sigma\bar M$.

\begin{theorem}[see~\cite{rst-57}]\label{T-2.1}
A weak $\beta f$-KM 
is locally isometric to a twisted product $\mathbb{R}^s\times_\sigma\bar M$, where $\bar M(\bar g, J)$ is a weak Hermitian manifold
$($a warped product if $\nabla\beta\perp{\cal D}$, and then $\bar M(\bar g, J)$ a weak K\"{a}hler manifold$)$ and
$-\beta\,\bar\xi$ is the mean curvature vector of the distribution ${\cal D}$.
Moreover, if $M$ is simply connected and complete, then the isometry is global.
\end{theorem}

\begin{example}\label{Ex-sHK}\rm
a) According to \cite{E-1923}, a weak K\"{a}hler~manifold with $J^2\ne -c^2\,{\rm id}$, where $c\in\mathbb{R}$, is reducible.
Take two (or even more) Hermitian manifolds $(\bar M_i,\bar g_i, J_i)$, hence $J_i^2=-{\rm id}_{\,i}$.
The product $\prod_{\,i}(\bar M_i,\bar g_i, c_i J_i)$, where $c_i\ne0$ are different constants, is a weak
Hermitian manifold with $Q=\bigoplus_{\,i}c_i^2{\rm id}_{\,i}$.
Moreover, if $(\bar M_i,\bar g_i, J_i)$ are K\"{a}hler~manifolds, then $\prod_{\,i}(\bar M_i,\bar g_i, c_i J_i)$ is a weak K\"{a}hler~manifold.
b)~Let $\bar M(\bar g,J)$ be a weak K\"{a}hler manifold and
$\sigma(t_1,\ldots,t_s)=c\,e^{\,\beta\sum t_i}$ a function on Euclidean space $\mathbb{R}^s$, where $c\ne0$ and $\beta$ are constants.
Then the warped product manifold $M=\mathbb{R}^s\times_\sigma\bar M$ has a weak metric $f$-structure which satisfies \eqref{2.3-f-beta}.
Using \eqref{4.NN} with $S=J$, for a weak K\"{a}hler manifold, we get $[{J},{J}]=0$; hence, ${\cal N}^{\,(1)}=0$ is true.
\end{example}

\begin{corollary}[see~\cite{rst-57}]
A weak $f$-KM
is locally a warped product $\mathbb{R}^s(t_1,\ldots,t_s)\times_\sigma \bar M$,
where $\sigma=c\,e^{\,\sum t_i}\ (c=const\ne0)$ and $\bar M(\bar g,J)$ is a weak K\"{a}hler manifold.
\end{corollary}

To simplify the calculations in the rest of the paper, we assume that $\beta=const\ne0$.



\begin{proposition}[see~\cite{rst-57}]
For a weak $\beta f$-KM $M^{2n+s}(f,Q,\xi_i,\eta^i,g)$
and $1\le i\le s$
we have
\begin{align}\label{2.4}
 & R_{X, Y}\,\xi_i = \beta^2\big\{\bar\eta(X)Y-\bar\eta(Y)X +\sum\nolimits_{\,j}\big(\bar\eta(Y)\eta^j(X) - \bar\eta(X)\eta^j(Y)\big)\xi_j\big\}
 ,\\
\label{2.5-f-beta}
 & \Ric^\sharp \xi_i = -2\,n\,\beta^2\bar\xi , \\
\label{3.1}
 & (\nabla_{\xi_i}\Ric^\sharp)X = - 2\,\beta\Ric^\sharp X - 4\,n\,\beta^3\big\{ s X - s\sum\nolimits_{\,j}\eta^j(X)\,\xi_j + \bar\eta(X)\,\bar\xi\,\big\}
 ,\\
\label{3.1A-f-beta}
 & \xi_i(r) = -2\,\beta\{r + 2\,s\,n(2\,n + 1)\,\beta^2\}.
\end{align}
\end{proposition}

Taking covariant derivative of \eqref{2.5-f-beta} along $X\in\mathfrak{X}_M$ and using \eqref{2.3b} gives
\begin{align}\label{E-L-02a}
 (\nabla_X\Ric^\sharp)\,\xi_i=
 - \beta\Ric^\sharp X - 2\,n\,\beta^3\big\{ s X - s\sum\nolimits_{\,j}\eta^j(X)\,\xi_j + \bar\eta(X)\,\bar\xi\,\big\} .
\end{align}
According to \eqref{2.5-f-beta}, weak $\beta f$-KM with $s>1$ cannot be Einstein manifolds.

\begin{proposition}[see~\cite{rst-57}]
For an $\eta$-Einstein \eqref{Eq-2.10} weak $\beta f$-KM $M^{2n+s}({f},Q,\xi_i,\eta^i,g)$
we obtain
\begin{align}\label{3.16}
 & {\rm Ric} = \big(s\beta^ 2+\frac{r}{2n}\big)\,\big\{g  - \sum\nolimits_{\,j}\eta^j\otimes\eta^j \big\}
 -2\,n\beta^2 \bar\eta\otimes\bar\eta .
\end{align}
\end{proposition}

The following theorem generalizes \cite[Theorem~1]{ghosh2019ricci} with $\beta\equiv1$ and $Q={\rm id}$.

\begin{theorem}[see~\cite{rst-57}]\label{Th-01}
Let a weak $\beta f$-KM $M^{2n+s}({f},Q,\xi_i,\eta^i,g)$
satisfy $\nabla_{\bar\xi}\Ric^\sharp=0$.
Then $(M,g)$ is an $\eta$-Einstein manifold \eqref{Eq-2.10} of scalar curvature $r=-2\,s\,n(2\,n+1)\,\beta^2$.
\end{theorem}


\section{The $\ast$-Ricci tensor}
\label{sec:03}

Here, we introduce and study the $\ast$-Ricci tensor for week metric $f$-manifolds  and give Example~\ref{Ex-ast-RS}.

We define $\ast$-{\em Ricci tensor} of a weak metric $f$-manifold $M^{2n+s}(f,Q,\xi_i,\eta^i,g)$ by the formula
\begin{align*}
 \Ric^\ast(X,Y)
 =(1/2)\,\tr\{Z\to f\,R_{X, fY} Z\}
 = (1/2)\,\tr\{Z\to R_{X, fY} fZ\}
  \quad (X,Y\in\mathfrak{X}_M),
\end{align*}
see also \eqref{E-ast-Ricci-tensor}.
Note that $\Ric^\ast$ is not symmetric.
The $\ast$-scalar curvature is given by
\[
 {r}^\ast=\tr_{g}\Ric^\ast=\sum\nolimits_{\,i}\Ric^\ast(e_i,e_i),
\]
where $e_i\ (1\le i\le 2n+s)$ is a local orthonormal basis of $TM$.

\begin{definition}\rm
A weak metric $f$-manifold is called $\ast$-$\eta$-{\em Einstein manifold}
if
\begin{align}\label{E-ast-Einstein}
 \Ric^\ast =
 \bar a\,g - \bar a\sum\nolimits_{\,i}\eta^i\otimes\eta^i + (\bar a+\bar b)\,\bar\eta\otimes\bar\eta\quad
 \mbox{for some}\ \ \bar a,\bar b\in C^\infty(M).
\end{align}
The formula $\Ric^\ast = \bar a\,g + \bar b\sum\nolimits_{\,i}\eta^i\otimes\eta^i + (\bar a+\bar b)\sum\nolimits_{\,i\ne j}\eta^i\otimes\eta^j$
is equivalent to
\eqref{E-ast-Einstein}.
\end{definition}

The following result generalizes Proposition~10 of \cite{kenmotsu1972class}.

\begin{proposition}
 On a weak $\beta f$-KM $M^{2n+s}({f},Q,\xi_i,\eta^i,g)$ we have
\begin{align}\label{E-lem2-1}
\notag
 & R_{X,Y}fZ - f R_{X,Y}Z = \beta^2\big\{ [s\,g(Y,Z) - s\sum\nolimits_{\,j}\eta^j(Y)\eta^j(Z) +\bar\eta(Y)\bar\eta(Z)]fX \\
\notag
 &\qquad -[s\,g(X,Z) - s\sum\nolimits_{\,j}\eta^j(X)\eta^j(Z) +\bar\eta(X)\bar\eta(Z)]fY \\
 &\qquad + g(X,fZ)\,[sY-s\sum\nolimits_{\,j}\eta^j(Y)\xi_j+\bar\eta(Y)\bar\xi]
  -g(Y,fZ)\,[sX-s\sum\nolimits_{\,j}\eta^j(X)\xi_j+\bar\eta(X)\bar\xi] \big\} ,\\
 \label{E-lem2-2}
  \notag
 & R_{fX, fY}Z - R_{X, Q Y}Z =
 \beta^2\Big\{
       \big\{g(Z,QY) - \sum\nolimits_{\,j}\eta^j(Y)\eta^j(Z)\big\}\{s\,X -s\sum\nolimits_{\,j}\eta^j(X)\xi_j + \bar\eta(X)\,\bar\xi\} \\
 \notag
 &\  -\{s\,g(Z,X)-s\sum\nolimits_{\,j}\eta^j(Z)\eta^j(X) + \bar\eta(Z)\bar\eta(X)\}\,\big\{ QY - \sum\nolimits_{\,j}\eta^j(Y)\xi_j\big\}\\
 \notag
 &\  +g(Z,fX)\big\{s\,fY {-} s\sum\nolimits_{\,j}\eta^j(Y)\xi_j {+} \bar\eta(Y)\,\bar\xi \big\}
 {-}\big\{s\,g(Z,fY) {-} s\sum\nolimits_{\,j}\eta^j(Z)\eta^j(Y) {+} \bar\eta(Z)\bar\eta(Y)\big\}fX \\
 &\  +\bar\eta(Y)\big[\bar\eta(Z)X - g(Z,X)\,\bar\xi + \sum\nolimits_{\,j}\eta^j(X)\{\eta^j(Z)\,\bar\xi -\bar\eta(Z)\,\xi_j\}\big] \Big\}.
\end{align}
\end{proposition}

\begin{proof} Recall the Ricci identity (commutation formula), for example \cite{CLN-2006}:
\begin{align*}
  \big(\nabla_X\nabla_Y f - \nabla_Y\nabla_X f -\nabla_{[X,Y]} f\big)Z = (R_{X,Y}f)Z = R_{X,Y} fZ - f R_{X,Y} Z.
\end{align*}
From this and \eqref{2.3-f-beta} we obtain \eqref{E-lem2-1}.
Using \eqref{E-lem2-1}, \eqref{2.2} and the skew-symmetry of $f$, we derive
\begin{align*}
 & g(R_{fZ, fW}X, Y) + g(R_{X,Y}Z, f^2W) = g(R_{X,Y}fZ - f R_{X,Y}Z, fW) \\
 & = \beta^2\Big\{
    [s\,g(Y,Z)-s\sum\nolimits_{\,j}\eta^j(Y)\eta^j(Z) + \bar\eta(Y)\bar\eta(Z)]\,\big\{g(X,QW) - \sum\nolimits_{\,j}\eta^j(W)\eta^j(X)\big\} \\
 & -[s\,g(X,Z)-s\sum\nolimits_{\,j}\eta^j(X)\eta^j(Z) + \bar\eta(X)\bar\eta(Z)]\,\big\{g(Y,QW) - \sum\nolimits_{\,j}\eta^j(W)\eta^j(Y)\big\}\\
 & +g(X,fZ)\big\{s\,g(Y,fW) - s\sum\nolimits_{\,j}\eta^j(Y)\eta^j(W) + \bar\eta(Y)\bar\eta(W) \big\}\\
 & -g(Y,fZ)\big\{s\,g(X,fW) - s\sum\nolimits_{\,j}\eta^j(X)\eta^j(W) + \bar\eta(X)\bar\eta(W)\big\}\Big\} .
\end{align*}
Using \eqref{2.1} and \eqref{2.4} we find
\begin{align*}
 & g(R_{X,Y}Z, f^2W) = -g(R_{X,Y}Z, Q W) -\sum\nolimits_{\,j} g(R_{X,Y}\xi_j, Z)\,\eta^j(W)
 = -g(R_{X,Y}Z, Q W) \\
 & - \beta^2\big\{\bar\eta(X)g(Y,Z)-\bar\eta(Y)g(X,Z)
 + \sum\nolimits_{\,j}\{\bar\eta(Y)\eta^j(X)-\bar\eta(X)\eta^j(Y)\}\,\eta^j(Z)\big\}\bar\eta(W) .
\end{align*}
Therefore
\begin{align*}
 & g(R_{fZ, fW}X - R_{Z, Q W}X,\ Y)  \\
 & = \beta^2\Big\{
       [s\,g(Y,Z)-s\sum\nolimits_{\,j}\eta^j(Y)\eta^j(Z) + \bar\eta(Y)\bar\eta(Z)]\,\big\{g(X,QW) - \sum\nolimits_{\,j}\eta^j(W)\eta^j(X)\big\} \\
 &\ \ -[s\,g(X,Z)-s\sum\nolimits_{\,j}\eta^j(X)\eta^j(Z) + \bar\eta(X)\bar\eta(Z)]\,\big\{g(Y,QW) - \sum\nolimits_{\,j}\eta^j(W)\eta^j(Y)\big\}\\
 &\ \ +g(X,fZ)\big\{s\,g(Y,fW) - s\sum\nolimits_{\,j}\eta^j(Y)\eta^j(W) + \bar\eta(Y)\bar\eta(W) \big\}\\
 &\ \ -g(Y,fZ)\big\{s\,g(X,fW) - s\sum\nolimits_{\,j}\eta^j(X)\eta^j(W) + \bar\eta(X)\bar\eta(W)\big\} \\
 &\ \ +\big[\bar\eta(X)g(Y,Z)-\bar\eta(Y)g(X,Z) + \sum\nolimits_{\,j}\{\bar\eta(Y)\eta^j(X)-\bar\eta(X)\eta^j(Y)\}\,\eta^j(Z)\big]\bar\eta(W) \Big\},
\end{align*}
or,
\begin{align*}
  R_{fZ, fW}X - R_{Z, Q W}X & =
 \beta^2\Big\{
       \big\{g(X,QW) - \sum\nolimits_{\,j}\eta^j(W)\eta^j(X)\big\}\{s\,Z -s\sum\nolimits_{\,j}\eta^j(Z)\xi_j + \bar\eta(Z)\,\bar\xi\} \\
 & -\{s\,g(X,Z)-s\sum\nolimits_{\,j}\eta^j(X)\eta^j(Z) + \bar\eta(X)\bar\eta(Z)\}\,\big\{ QW - \sum\nolimits_{\,j}\eta^j(W)\xi_j\big\}\\
 & +g(X,fZ)\big\{s\,fW - s\sum\nolimits_{\,j}\eta^j(W)\xi_j + \bar\eta(W)\,\bar\xi \big\}\\
 & -\big\{s\,g(X,fW) - s\sum\nolimits_{\,j}\eta^j(X)\eta^j(W) + \bar\eta(X)\bar\eta(W)\big\}fZ \\
 & +\bar\eta(W)\big[\bar\eta(X)Z - g(X,Z)\,\bar\xi + \sum\nolimits_{\,j}\eta^j(Z)\{\eta^j(X)\,\bar\xi -\bar\eta(X)\,\xi_j\}\big] \Big\}.
\end{align*}
From the above we find \eqref{E-lem2-2}.
\end{proof}

For a weak metric $f$-manifold $M^{2n+s}(f,Q,\xi_i,\eta^i,g)$,
we will build an $f$-\textit{basis},
consis\-ting of mutually orthogonal nonzero vectors
at a point $x\in M$.
Let $e_1\in(\ker\eta)_x$ be a unit eigenvector of the self-adjoint operator $Q>0$ with the eigenvalue $\lambda_1>0$. Then, $fe_1\in(\ker\eta)_x$ is orthogonal to $e_1$
and $Q(fe_1) = f(Qe_1) = \lambda_1 fe_1$.
Thus, the subspace orthogonal to the plane $span\{e_1,fe_1\}$ is $Q$-invariant.
 There exists a unit vector $e_2\in(\ker\eta)_x$ such that $e_2\perp span\{e_1, fe_1\}$
and $Q\,e_2= \lambda_2 e_2$  for some $\lambda_2>0$.
Obviously, $Q(fe_2) = f(Q\,e_2) = \lambda_2 fe_2$.
All five vectors $\{e_1, fe_1,e_2, fe_2, \xi_1,\ldots,\xi_s\}$ are nonzero and mutually orthogonal.
Continuing in the same manner, we find a basis $\{e_1, fe_1,\ldots, e_n, fe_n, \xi_1,\ldots,\xi_s\}$ of $T_x M$
consisting of mutually orthogonal vectors.
Note that $g(fe_i,fe_i)=g(Qe_i,e_i)=\lambda_i$.

\begin{theorem}\label{T-1-ast}
For a weak $\beta f$-KM $M^{2n+s}(f,Q,\xi_i,\eta^i,g)$, the $\ast$-{Ricci tensor}
and the $\ast$-{scalar curvature}
are related to the Ricci tensor and the scalar curvature
by the formulas
\begin{align}\label{E-ast-Ric}
 \Ric^\ast(X,Y) &{=} \Ric(X,QY){+}\beta^2\big\{s(2\,n{-}1)\,g(QX,Y){+}2\,n\bar\eta(X)\bar\eta(Y){-}s(2\,n{-}1)\sum\nolimits_{\,j}\eta^j(X)\eta^j(Y)\big\},\\
\label{E-ast-scal}
 {r}^\ast &= \tr(Q\Ric^\sharp) + \beta^2\{4s\,n^2 + s\,(2\,n-1)\tr\widetilde Q\}.
\end{align}
\end{theorem}

\begin{proof}
Using the first Bianchi identity, we obtain
\begin{align}\label{E-th1-a}
 2\Ric^\ast(X,Y) = \tr\{Z\to -f\,R_{fY,Z}X\} + \tr\{Z\to -f\,R_{Z, X} fY\}.
\end{align}
Let $e_i\ (1\le i\le 2n+s)$ be a local $f$-basis of $TM$ and $e_{2n+j}=\xi_j$.
Then $e'_i=fe_i/\|fe_i\|\ (1\le i\le 2n)$ is a local orthonormal basis of ${\cal D}$.
Using \eqref{2.1}, \eqref{2.4}, the skew-symmetry of $f$ and the symmetries of the
curvature tensor, we show that the two terms in the right-hand side of \eqref{E-th1-a} are equal:
\begin{align}\label{E-th1-b}
\notag
 & \ \tr\{Z\to -f\,R_{fY,Z}X\}
 = -\sum\nolimits_{\,i} g(f R_{fY, e'_i}X, e'_i)
 = \sum\nolimits_{\,i} g(R_{fY, fe_i/\|fe_i\|}X, f^2 e_i/\|fe_i\|) \\
 \notag
 & = \sum\nolimits_{\,i} g(R_{fY, fe_i/\|fe_i\|}X,
 -\lambda_i e_i/\|fe_i\|)
 - \sum\nolimits_{\,i} \eta(e_i)/\|fe_i\|
 g(R_{fY, fe_i/\|fe_i\|}\,\xi, X) \\
 & = -\sum\nolimits_{\,i} (\lambda_i/\|fe_i\|^2)g(R_{fY, fe_i}X, e_i)
 = \sum\nolimits_{\,i} g(R_{e_i,X}fY, fe_i)
 = \tr\{Z\to -f\,R_{Z, X} fY\}.
\end{align}
Using \eqref{E-lem2-2} with $Z=e_i$,
\begin{align*}
 \notag
 & -\tr\{Z\to -f\,R_{Z, W} fY\} - \Ric(Q Y,W) = \sum\nolimits_{\,i}\{g(R_{fe_i, fY}e_i - R_{e_i, Q Y}e_i,\ W)\}  \\
 & = \beta^2\sum\nolimits_{\,i}\Big\{
       [s\,g(W,e_i)-s\sum\nolimits_{\,j}\eta^j(W)\eta^j(e_i) + \bar\eta(W)\bar\eta(e_i)]\,\big\{g(e_i,QY) - \sum\nolimits_{\,j}\eta^j(Y)\eta^j(e_i)\big\} \\
 &\ \ -[s\,g(e_i,e_i)-s\sum\nolimits_{\,j}\eta^j(e_i)\eta^j(e_i) + \bar\eta(e_i)\bar\eta(e_i)]\,\big\{g(W,QY) - \sum\nolimits_{\,j}\eta^j(Y)\eta^j(W)\big\}\\
 &\ \ +g(e_i,fe_i)\big\{s\,g(W,fY) - s\sum\nolimits_{\,j}\eta^j(W)\eta^j(Y) + \bar\eta(W)\bar\eta(Y) \big\}\\
 &\ \ -g(W,fe_i)\big\{s\,g(e_i,fY) - s\sum\nolimits_{\,j}\eta^j(e_i)\eta^j(Y) + \bar\eta(e_i)\bar\eta(Y)\big\} \\
 &\ \ +\big[\bar\eta(e_i)g(W,e_i)-\bar\eta(W)g(e_i,e_i) + \sum\nolimits_{\,j}\{\bar\eta(W)\eta^j(e_i)-\bar\eta(e_i)\eta^j(W)\}\,\eta^j(e_i)\big]\bar\eta(Y) \Big\} \\
 & = \beta^2\sum\nolimits_{\,i}\Big\{ - s(2\,n-1)\,g(QY,W) + s(2n-1)\sum\nolimits_{\,j}\eta^j(Y)\eta^j(W) - 2n\,\bar\eta(Y)\bar\eta(W) \Big\},
\end{align*}
where $e_i\ (1\le i\le 2n+s)$ is a local orthonormal basis of $TM$ with $e_{\,2n+j}=\xi_j\ (1\le j\le s)$,
we get
\begin{align*}
\notag
 & \tr\{Z\to -f\,R_{Z, W} fY\} +  \Ric(Q Y,W) \\
 & = \beta^2\sum\nolimits_{\,i}\Big\{ s(2\,n-1)\,\big[g(QY,W) - \sum\nolimits_{\,j}\eta^j(Y)\eta^j(W)\big] + 2\,n\,\bar\eta(Y)\bar\eta(W) \Big\} ,
\end{align*}
that, in view of \eqref{E-th1-b}, completes the proof of \eqref{E-ast-Ric}.
Contracting \eqref{E-ast-Ric} yields \eqref{E-ast-scal}.
\end{proof}

Using Theorem~\ref{T-1-ast}, we can express $\Ric^\ast$ in terms of ${r}^\ast$ as follows.

\begin{corollary}
On a $\ast$-$\eta$-Einstein weak $\beta f$-KM the expression of $\Ric^{\ast}$ is the following:
\begin{align*}
 \Ric^\ast = \frac{{r}^\ast}{2\,n}\Big\{g - \sum\nolimits_{\,j}\eta^j\otimes\eta^j\Big\} .
\end{align*}
\end{corollary}

\begin{proof}
Tracing \eqref{E-ast-Einstein} gives ${r}^\ast=(2n+s)\bar a+s\,\bar b$. Putting $X=Y=\xi_i$ in \eqref{E-ast-Einstein} and using \eqref{E-ast-Ric} yields
$\bar a+\bar b=0$. From the above equalities, we express parameters of \eqref{E-ast-Einstein} as follows: $\bar a=-\bar b=\frac{{r}^\ast}{2\,n}$.
\end{proof}

The following example of $\ast$-$\eta$-RS on a weak $\beta f$-KM will illustrate our main results.

\begin{example}[see \cite{rst-58} for $s=\beta=1$]\label{Ex-ast-RS}\rm
We consider linearly independent vector fields
 $e_i = e^{-\beta\bar x}\partial_i$, $e_{2n+p} = \partial_{2n+p}\ (1\le i\le 2n,\ 1\le p\le s)$
in $M=\mathbb{R}^{2n+s}(x_1,\ldots,x_{2n+s})$, where $\bar x=\sum_{p=1}^s x_{2n+p}$ and $\beta=const\ne0$,
and define a Riemannian metric
as $g(e_i,e_j)=\delta_{ij}\ (1\le i,j\le 2n+s)$. Set $\xi_p=e_{2n+p}$ and $\eta^p=dx_{2n+p}$ for $1\le p\le s$.
Given~$c=const\ge 0$, define (1,1)-tensors $f$ and $Q$~by
\begin{align*}
& f e_i = \sqrt{1+c}\,e_{n+i},\quad
  f e_{n+i} = -\sqrt{1+c}\,e_i,\quad
  f e_{2n+p} = 0, \\
& Q e_j = (1+c)\,e_j,\quad
  Q e_{2n+p} = e_{2n+p}\quad(1\le i\le n,\ 1\le j\le 2n,\ 1\le p\le s).
\end{align*}
Since \eqref{2.1}--\eqref{2.2} are valid, $M(f,Q,\xi_i,\eta^i,g)$ is a weak metric $f$-manifold when $c>0$ (a metric $f$-manifold when $c=0$).
We can deduce that
\[
 [e_i,e_j]=0,\quad
 [e_i, e_{2n+p}] = \beta e_i,\quad
 [e_{2n+p},e_{2n+q}]= 0\quad (1\le i,j\le 2n,\ 1\le p,q\le s).
\]
For the Levi-Civita connection we get
$2g(\nabla_{e_i}e_j, e_k)= g(e_k,[e_i,e_j])-g(e_i,[e_j,e_k])-g(e_j,[e_i,e_k])$, thus
\begin{align*}
 \nabla_{e_i}e_j = -\delta_{ij}\,\beta\,\bar\xi,\ \
 \nabla_{e_j}e_{2n+p} = \beta e_j,\ \
 \nabla_{e_{2n+p}}\,e_{2n+q} = 0\ \
 (1\le i\le 2n+s,\ 1\le j\le 2n,\ 1\le p,q\le s).
\end{align*}
Since \eqref{2.3-f-beta} is valid, $M(f,Q,\xi_i,\eta^i,g)$ is a weak $\beta f$-KM when $c>0$ ($\beta f$-KM when $c=0$).
The~distributions ${\cal D}={\rm Span}(e_1,\ldots,e_{2n})$ and ${\cal D}^\bot={\rm Span}(e_{2n+1},\ldots,e_{2n+s})$ are involutive;
${\cal D}^\bot$ defines a foliation with flat totally geodesic leaves and ${\cal D}$ defines a totally umbilical foliation.
The second fundamental form of ${\cal D}$ is $h(e_i,e_j)=-\beta\delta_{ij}\,\bar\xi\ (1\le i,j\le 2\,n)$
and the mean curvature vector is $H=-\beta\,\bar\xi$.

The non-vanishing components of the curvature tensor are
\begin{align*}
 & R_{\,e_i,e_j}e_j = - R_{\,e_j,e_i}e_j = - s\,\beta^2 e_i,\quad
 R_{\,e_i,e_{2n+p}}e_{2n+q}
 = -\beta^2 e_i,\\
 & R_{\,e_{2n+p},e_j}e_j = -\beta^2\,\bar\xi \quad (1\le i,j\le 2n,\ 1\le p,q\le s).
\end{align*}
From the above results we have the Ricci (1,1)-tensor
\begin{align*}
 \Ric^\sharp e_i
 = -2\,s\,n\beta^2 e_i,\quad \Ric^\sharp e_{2n+p} = -2\,n\beta^2 \bar\xi\quad (1\le i\le 2n,\ 1\le p\le s).
\end{align*}
Therefore, $g$ is an $\eta$-Einstein \eqref{Eq-2.10} metric
with
$a=-2 s\,n\beta^2$ and $b=2 (s-1)\,n\beta^2$.
Contracting $\Ric$
we find the scalar curvature $r=-2\,s\,n(2\,n+1)\beta^2$.
Next we derive
\[
 \Ric^\ast(e_i,e_j)=-(1+c)\big(2n(s-1)+1\big)\delta_{ij}\beta^2,\quad
 \Ric^\ast(e_{2n+p},e_{2n+q})=0,\quad
 \Ric^\ast(e_i,e_{2n+p})=0
\]
for all $1\le i,j\le 2n,\ 1\le p,q\le s$.
Hence $g$ is a $\ast$-$\eta$-Einstein metric,
see \eqref{E-ast-Einstein} with $\bar a=-\bar b=-(1+c)\big(2n(s-1)+1\big)\beta^2$, of $\ast$-scalar curvature $r^*=-2n\big(2n(s-1)+1\big)(1+c)\beta^2$.
\end{example}

\section{Main Results}
\label{sec:04}

Here, we study the interaction of $\ast$-$\eta$-RS with weak $\beta f$-KM.
Proposition~\ref{T-lambda0} derives the sum of the soliton constants.
Theorems~\ref{T-002a}--\ref{T-005} present new characteristics of $\eta$-Einstein~metrics.

\begin{definition}\label{D-ast-eta-RS}\rm
Given $\mu,\lambda\in\mathbb{R}$ and a vector field $V$
on a weak metric $f$-manifold $M^{2n+s}({f},Q,{\xi_i},{\eta^i},g)$,
we define a $\ast$-$\eta$-{\em RS} be the formula
\begin{align}\label{E-sol-ast-eta}
 (1/2)\pounds_V\,g + \Ric^\ast = \lambda\,\big\{g -\sum\nolimits_{\,i}\eta^i\otimes\eta^i\big\} +(\lambda+\mu)\,\bar\eta\otimes\bar\eta .
\end{align}
A
$\ast$-$\eta$-RS
is called \textit{expanding}, \textit{steady}, or \textit{shrinking} if $\,\lambda$, is negative, zero, or positive, respectively.
\end{definition}

\begin{remark}\rm
For a Killing vector field $V$,
i.e., $\pounds_V\,g=0$,
equation \eqref{E-sol-ast-eta} gives a $\ast$-$\eta$-Einstein metric with constant parameters, see \eqref{Eq-2.10}.
For $s=1$,
\eqref{Eq-2.10} and \eqref{E-sol-ast-eta} give the definitions for a weak almost contact metric~manifold, see~\cite{rst-57}:
from~\eqref{Eq-2.10} we get
an $\ast$-$\eta$-Einstein structure
$\Ric^\ast = a\,g + b\,\eta\otimes\eta$,
and \eqref{E-sol-ast-eta} gives an $\ast$-$\eta$-RS $\frac12\,\pounds_V\,g + \Ric^\ast = \lambda\,g + \mu\,\eta\otimes\eta$.
\end{remark}

Note that $\Ric^\ast$ in \eqref{E-sol-ast-eta}
is necessarily symmetric.
If $V$ is a vector field Killing, i.e., $\pounds_V\,g=0$, then solitons of \eqref{E-sol-ast-eta}
are trivial and $g$ becomes an $\ast$-$\eta$-{\em Einstein} metric, given by \eqref{E-sol-ast-eta0} for $s=1$.

\begin{example}
\rm
Let $M(f,Q,\xi_i,\eta^i,g)$ be the weak $\beta f$-KM (with $c\ne0$) of Example~\ref{Ex-ast-RS}.
We~can justify that $\pounds_{\bar\xi}\,g=2\,\beta s\{g-\sum\nolimits_{\,j}\eta^j\otimes\eta^j\}$ holds.
Thus \eqref{E-sol-ast-eta} is valid for $V=\bar\xi$ and $\lambda=-\mu=s\beta-(1+c)(2\,n(s-1)+1)\beta^2$.
Thus $(g,\bar\xi)$ represents a $\ast$-$\eta$-RS on $M$.
\end{example}

For a weak $\beta f$-KM, using
the formula of $\ast$-Ricci tensor \eqref{E-ast-Ric} in the $\ast$-$\eta$-RS equation \eqref{E-sol-ast-eta}, we~get
\begin{align}\label{E-sol-eta2}
\notag
 & (1/2)\pounds_V\,g(X,Y) + \Ric(X,QY) = \lambda\,g(X,Y) - s\,(2\,n-1)\beta^2 g(QX,Y) \\
 & +(s(2\,n-1)\beta^2-\lambda)\sum\nolimits_{\,j}\eta^j(X)\eta^j(Y) + (\lambda+\mu-2\,n\beta^2)\bar\eta(X)\bar\eta(Y)  ;
 \end{align}
in this case, $\Ric(X,QY)=\Ric(Y,QX)$ holds, thus $Q$ commutes with the Ricci operator $\Ric^\sharp$.

For $s=1$, \eqref{E-sol-eta2} gives the following, see \cite{rst-57}:
\begin{align*}
 (1/2)(\pounds_V g)(X,Y) + \Ric(X,QY) = \lambda\,g(X,Y) - (2n-1)\beta^2 g(X,QY) + (\mu-\beta^2)\eta(X)\eta(Y) .
\end{align*}

Developing the method of Lemma 3 in \cite{ven-2022}, we obtain the following.

\begin{lemma}
Let $(g,V)$ represent a $\ast$-$\eta$-RS \eqref{E-sol-ast-eta} on a weak $\beta f$-KM $M^{2n+s}(\varphi,Q,\xi_i,\eta^i,g)$. Then
\begin{align}\label{3.9}
\notag
 & (\pounds_{V} R)_{X,Y} \xi_i = 2\,\beta\{ (\nabla_X\Ric^\sharp)(Q Y) - (\nabla_Y\Ric^\sharp)(Q X)\} \\
\notag
 & + 2\,\beta^2\{\bar\eta(X)\Ric^\sharp Y - \bar\eta(Y)\Ric^\sharp X \}
 + 4\,\beta^2\{ \bar\eta(X)\Ric^\sharp \widetilde QY - \bar\eta(Y)\Ric^\sharp \widetilde QX\} \\
 & + 4\,s\,n\beta^4\bar\eta(X)\big\{ Y + 2\,\widetilde QY - \sum\nolimits_{\,k}\eta^k(Y)\xi_k \big\}
 - 4\,s\,n\beta^4\bar\eta(Y)\big\{ X + 2\,\widetilde QX - \sum\nolimits_{\,k}\eta^k(X)\xi_k \big\}
\end{align}
for all $X,Y\in\mathfrak{X}_M$ and $1\le i\le s$.
Furthermore, using $Y=\xi_j$ in \eqref{3.9} gives
\begin{align}\label{E-3.18}
 (\pounds_V R)_{X,\,\xi_j}\,\xi_i = 0 \quad (1\le i,j\le s,\ X\in\mathfrak{X}_M).
\end{align}
\end{lemma}

\begin{proof}
Taking the covariant derivative of \eqref{E-sol-eta2} along $Z\in\mathfrak{X}_M$ and using \eqref{2.3-f-beta}, we~get
\begin{align}\label{3.3A-patra}
\notag
 (\nabla_Z\,\pounds_{V}\,g)(X,Y) &= -2(\nabla_{Z}\,{\rm Ric})(X,QY) -2\Ric(X,(\nabla_Z\,Q)Y) -2s(2n-1)\beta^2 g(X,(\nabla_Z\,Q)Y)\\
\notag
 & + 2\,\beta\{s\,\mu+(s-1)\lambda -s\beta^2)\}\big\{ \big[g(X,Z) -\sum\nolimits_{\,j}\eta^j(X)\eta^j(Z)\big]\bar\eta(Y) \\
 & + \big[g(Y,Z) -\sum\nolimits_{\,j}\eta^j(Y)\eta^j(Z)\big]\bar\eta(X)\big\}
\end{align}
for all $X,Y\in\mathfrak{X}_M$.
Recall the commutation formula with the tensor $\pounds_{V}\nabla$, see \cite[p.~23]{yano1970integral},
\begin{align}\label{2.6}
 (\pounds_{V}\nabla_{Z}\,g - \nabla_{Z}\,\pounds_{V}\,g - \nabla_{[V,Z]}\,g)(X,Y) = -g((\pounds_{V}\nabla)(Z,X),Y) -g((\pounds_{V}\nabla)(Z,Y),X).
\end{align}
Since Riemannian metric is parallel, $\nabla g=0$, it follows from \eqref{2.6} that
\begin{equation}\label{3.4}
 (\nabla_{Z}\,\pounds_{V}\,g)(X,Y) = g((\pounds_{V}\nabla)(Z,X),Y) + g((\pounds_{V}\nabla)(Z,Y),X).
\end{equation}
Since $(\pounds_{V}\nabla)(X,Y)$ is symmetric, from \eqref{3.4} we get
\begin{align}\label{3.5cycle}
 2\,g((\pounds_{V}\nabla)(X,Y), Z) =
  (\nabla_X\,\pounds_{V}\,g)(Y,Z)
  +(\nabla_Y\,\pounds_{V}\,g)(Z,X)
  -(\nabla_Z\,\pounds_{V}\,g)(X,Y) .
\end{align}
According to \eqref{3.5cycle} and \eqref{3.3A-patra},
we obtain
\begin{align}\label{3.6}
\nonumber
 & g((\pounds_{V}\nabla)(X,Y),Z) = (\nabla_{Z}\,{\rm Ric})(X,QY) -(\nabla_{X}\,{\rm Ric})(Y,QZ) -(\nabla_{Y}\,{\rm Ric})(Z,QX)\\
 & + 2\,\beta\{s\,\mu+(s-1)\lambda -s\,\beta^2)\}\,\bar\eta(Z)\{g(X,Y) -\sum\nolimits_{\,j}\eta^j(X)\,\eta^j(Y) \} +\varepsilon(X,Y,Z),
\end{align}
where
\begin{align*}
 \varepsilon(X,Y,Z)=
 \Ric(X,(\nabla_Z\,Q)Y)+s(2\,n-1)\beta^2 g(X,(\nabla_Z\,Q)Y) \\
-\Ric(Z,(\nabla_Y\,Q)X)-s(2\,n-1)\beta^2 g(Z,(\nabla_Y\,Q)X) \\
-\Ric(Y,(\nabla_X\,Q)Z)-s(2\,n-1)\beta^2 g(Y,(\nabla_X\,Q)Z).
\end{align*}
Substituting $Y=\xi_i$ in \eqref{3.6}, we obtain
\begin{align}\label{E-Lem3-a}
 g((\pounds_{V}\nabla)(X,\xi_i),Z) =(\nabla_{Z}\,{\rm Ric})(X,\xi_i)
 {-}(\nabla_{X}\,{\rm Ric})(\xi_i,QZ)
 {-}(\nabla_{\xi_i}\,{\rm Ric})(Z,QX) + \varepsilon(X,\xi_i,Z) ,
\end{align}
where
in view of \eqref{E-nS-10b}, we have
\begin{align*}
 \varepsilon(X,\xi_i,Z) = - \beta\Ric(\widetilde Q X, Z) - 2\,s\,n\beta^3 g(\widetilde Q X, Z).
\end{align*}
Applying \eqref{E-L-02a} and \eqref{3.1} to \eqref{E-Lem3-a}, we obtain
\begin{align}\label{3.7}
 (\pounds_{V} \nabla)(X,\xi_i) = 2\,\beta\Ric^\sharp Q X + 4s\,n\beta^3 Q X + 2\,s\,\beta^3 \widetilde Q X
 + 4\,n\beta^3 \big\{ \bar\eta(X)\bar\xi - s\sum\nolimits_{\,j}\eta^j(X)\xi_j \big\} .
\end{align}
Next, using \eqref{2.3b} and \eqref{2.3c} in the covariant derivative of (\ref{3.7}) for $Y\in\mathfrak{X}_M$, yields for any $X\in\mathfrak{X}_M$:
\begin{align*}
 &(\nabla_Y(\pounds_V\nabla))(X,\xi_i) = - \beta(\pounds_V\nabla)(X,Y) + 2\,\beta(\nabla_Y\Ric^\sharp)(Q X)
 {+} 2\,\beta^2\bar\eta(Y)\Ric^\sharp QX  {-} 2\,\beta^2\bar\eta(X)\Ric^\sharp \widetilde QY \\
 & + 4\,n\beta^4\bar\eta(Y)\big\{ s\,QX + \bar\eta(X)\,\bar\xi - s\sum\nolimits_{\,j}\eta^j(X)\xi_j \big\}
   - 4\,n\beta^4\{s\,\bar\eta(X)\widetilde QY + (s-1)\,g(\widetilde Q X, Y)\bar\xi\}.
\end{align*}
Recall the formula, e.g., \cite[Eq. (7)]{ghosh2019ricci} and  \cite[p.~23]{yano1970integral}):
\begin{align}\label{Eq-7}
 & (\pounds_V\,R)_{X,Y} Z = (\nabla_X(\pounds_V\nabla))(Y,Z) - (\nabla_Y(\pounds_V\nabla))(X,Z).
\end{align}
Plugging this in the formula \eqref{Eq-7} with $Z=\xi_i$,
and using symmetry of $(\pounds_V\nabla)(X,Y)$, we deduce \eqref{3.9}.
Substituting $Y=\xi_j$ in \eqref{3.9} and using \eqref{2.5-f-beta}, \eqref{3.1} and \eqref{E-L-02a}, gives
\begin{align*}
  (\pounds_{V} R)_{X,\,\xi_j}\,\xi_i & = 2\,\beta\{ (\nabla_X\Ric^\sharp)\xi_j - (\nabla_{\xi_j}\Ric^\sharp)Q X\}
  + 2\,\beta^2\{\bar\eta(X)\Ric^\sharp \xi_j - \Ric^\sharp QX \} \\
 & - 2\,\beta^2\Ric^\sharp \widetilde QX - 4 s\,n\beta^4\big\{ X + 2\,\widetilde QX - \sum\nolimits_{\,k}\eta^k(X)\xi_k \big\} \\
 & = 2\,\beta^2\Big\{ - \Ric^\sharp X - 2\,n\,\beta^2\big\{ s X - s\sum\nolimits_{\,k}\eta^k(X)\,\xi_k + \bar\eta(X)\,\bar\xi\,\big\}\Big\} \\
 & -4\beta^2\{ -\Ric^\sharp QX - 2\,n\beta^2[s QX - s\sum\nolimits_{\,k}\eta^k(X)\,\xi_k + \bar\eta(X)\,\bar\xi] \} \\
 & - 2\,\beta^2\{2\,n\beta^2\bar\eta(X)\bar\xi + \Ric^\sharp QX \} - 2\beta^2 \Ric^\sharp\widetilde QX \\
 & -4 s\,n\beta^4\{ X + 2\,\widetilde QX - \sum\nolimits_{\,k}\eta^k(X)\,\xi_k \} = 0 ,
\end{align*}
that is \eqref{E-3.18}.
\end{proof}

The following result generalizes Theorem~3.1 of \cite{ven-2019}.

\begin{proposition}\label{T-lambda0}
Let $(g,V)$ represent a $\ast$-$\eta$-RS \eqref{E-sol-ast-eta} on a weak $\beta f$-KM $M^{2n+s}(\varphi,Q,\xi_i,\eta^i,g)$.
Then the soliton constants satisfy
 $\lambda +\mu = 0$.
\end{proposition}

\begin{proof}
Using \eqref{2.4} and $g(R_{X,\,\xi_i}Z, W)=-g(R_{Z,W}\,\xi_i,X)$, we derive
\begin{align}\label{E-3.19a}
 R_{X,\xi_i}Z = \beta^2\big\{ g(X,Z)\bar\xi - \bar\eta(Z) X + \sum\nolimits_{\,j}\eta^j(X)\big(\bar\eta(Z)\xi_j - \eta^j(Z)\,\bar\xi \big) \big\}.
\end{align}
Taking the Lie derivative along $V$ of
 $R_{X,\,\xi_j}\,\xi_i=\beta^2\big\{\sum\nolimits_{\,k}\eta^k(X)\xi_k - X\big\}$,
see \eqref{2.4} with $Y=\xi_j$ (or \eqref{E-3.19a} with $Z=\xi_j$), and using \eqref{2.4} and \eqref{E-3.19a},
gives
\begin{align}\label{E-3.19ab}
\notag
 & (\pounds_V  R)_{X,\,\xi_j}\,\xi_i = - \beta^2\Big\{\bar\eta(X)\pounds_V\,\xi_j -\bar\eta(\pounds_V\,\xi_j)X
 + \sum\nolimits_{\,k}\{\bar\eta(\pounds_V\,\xi_j)\,\eta^k(X) -\bar\eta(X)\eta^k(\pounds_V\,\xi_j)\}\xi_k \Big\} \\
\notag
 & - \beta^2\Big\{g(X,\pounds_V\,\xi_i)\,\bar\xi
 -\bar\eta(\pounds_V\,\xi_i)X + \sum\nolimits_{\,k}\eta^k(X)\{\bar\eta(\pounds_V\,\xi_i)\,\xi_k -\eta^k(\pounds_V\,\xi_i)\,\bar\xi\}\Big\} \\
 & + \beta^2\sum\nolimits_{\,k}\big\{ (\pounds_V\,\eta^k)(X)\xi_k + \eta^k(X)\pounds_V\,\xi_k\big\}.
\end{align}
Here we used $\pounds_V  (R_{X,\,\xi_j}\,\xi_i) =(\pounds_V  R)_{X,\,\xi_j}\,\xi_i + R_{X,\pounds_V\,\xi_j}\,\xi_i + R_{X,\,\xi_j}\,\pounds_V\,\xi_i$ with
\begin{align*}
 & R_{X,\pounds_V\,\xi_j}\,\xi_i = \beta^2\big\{\bar\eta(X)\pounds_V\,\xi_j -\bar\eta(\pounds_V\,\xi_j)X
 + \sum\nolimits_{\,k}\{\bar\eta(\pounds_V\,\xi_j)\,\eta^k(X) -\bar\eta(X)\eta^k(\pounds_V\,\xi_j)\}\xi_k \big\},\\
 & R_{X,\,\xi_j}\,\pounds_V\,\xi_i =\beta^2\big\{g(X,\pounds_V\,\xi_i)\,\bar\xi
 -\bar\eta(\pounds_V\,\xi_i)X +\sum\nolimits_{\,k}\eta^k(X)\{\bar\eta(\pounds_V\,\xi_i)\,\xi_k -\eta^k(\pounds_V\,\xi_i)\,\bar\xi\}\big\}.
\end{align*}
In view of \eqref{E-3.18}, the equation \eqref{E-3.19ab} divided by $\beta^2$ becomes
\begin{align}\label{E-3.19}
\notag
 & \sum\nolimits_{\,k}(\pounds_V\,\eta^k)(X)\xi_k + \sum\nolimits_{\,k}\eta^k(X)\pounds_V\,\xi_k
 - \bar\eta(X)\pounds_V\,\xi_j + \bar\eta(\pounds_V\,\xi_j)X \\
\notag
 & - \bar\eta(\pounds_V\,\xi_j)\sum\nolimits_{\,k}\eta^k(X)\xi_k +\bar\eta(X)\sum\nolimits_{\,k}\eta^k(\pounds_V\,\xi_j)\xi_k
 - g(X,\pounds_V\,\xi_i)\,\bar\xi +\bar\eta(\pounds_V\,\xi_i)X  \\
 & - \bar\eta(\pounds_V\,\xi_i)\sum\nolimits_{\,k}\eta^k(X)\,\xi_k +\sum\nolimits_{\,k}\eta^k(X)\eta^k(\pounds_V\,\xi_i)\,\bar\xi
 = 0 .
\end{align}
For $X\in{\cal D}$
the equation \eqref{E-3.19} reduces to the following:
\begin{align}\label{E-3.19b}
 \sum\nolimits_{\,k}(\pounds_V\,\eta^k)(X)\xi_k  + \bar\eta(\pounds_V\,\xi_i)X + \bar\eta(\pounds_V\,\xi_j)X - g(X,\pounds_V\,\xi_i)\,\bar\xi  = 0 .
\end{align}
Taking the ${\cal D}$- and ${\cal D}^\bot$- components of \eqref{E-3.19b} with $i=j$ yields
\begin{align}\label{E-3.19c}
 \bar\eta(\pounds_V\,\xi_i)=0 ,\quad
 (\pounds_V\,\eta^k)(X) = g(X,\pounds_V\,\xi_i)\quad (X\in\mathfrak{X}_M,\ 1\le k\le s).
\end{align}
Using \eqref{2.5-f-beta}, we write \eqref{E-sol-eta2} with $Y=\xi_k$ as
\begin{align}\label{E-3.20}
\notag
 & (\pounds_V\,g)(X,\xi_k) = 4\,n\beta^2\bar\eta(X) + 2\,\lambda\,\eta^k(X) -2\,s\,(2\,n-1)\beta^2\eta^k(X) \\
 & + 2\,\{s(2\,n-1)\beta^2-\lambda\}\eta^k(X)  + 2(\lambda+\mu - 2\,n\beta^2)\,\bar\eta(X) = 2\,(\lambda+\mu)\,\bar\eta(X) .
\end{align}
Using the equality
\[
 (\pounds_V\,g)(\xi_i,\xi_k) =-g(\xi_i, \pounds_V\,\xi_k)-g(\xi_k, \pounds_V\,\xi_i)=-\eta^i(\pounds_V\,\xi_k)-\eta^k(\pounds_V\,\xi_i),
\]
the equation \eqref{E-3.20} for $X=\xi_i$ reduces to
\begin{align}\label{E-3.21}
 \eta^i(\pounds_V\,\xi_k) +\eta^k(\pounds_V\,\xi_i) = -2(\lambda +\mu) \ \Rightarrow\
  \bar\eta(\pounds_V\,\bar\xi) = -s^2(\lambda +\mu).
\end{align}
Comparing the above equality with the first equation of \eqref{E-3.19c} completes the proof.
\end{proof}

\begin{corollary}
Let $(g,V)$ represent a $\ast$-RS on a weak $\beta f$-KM $M^{2n+s}(\varphi,Q,\xi_i,\eta^i,g)$.
Then the soliton constant satisfies $\lambda=0$.
\end{corollary}

The following result generalizes Theorem~3 of \cite{rst-58} where $s=1$.

\begin{theorem}\label{T-002}
Let $M^{2n+s}({f},Q,\xi_i,\eta^i,g)$ be a weak $\eta$-Einstein \eqref{Eq-2.10} $\beta f$-KM.
Suppose that $(g,V)$ represents a $\ast$-$\eta$-RS \eqref{E-sol-ast-eta},
and if~$s>1$, then we assume $V\in{\cal D}$.
Then
$a=-2\,s\,n\beta^2$, $b=2(s-1)n\beta^2$ and
${r}=-2 s\,n(2\,n+1)\beta^2$;
moreover, if $s=1$ then $(M,g)$ is an Einstein manifold.
\end{theorem}

\begin{proof}
Set $s>1$. Taking covariant derivative of \eqref{3.16} along $Y$ and using \eqref{2.3b}, we~get
\begin{align}\label{E-Th2-a}
\nonumber
 & (\nabla_Y \Ric^\sharp)X = \frac{Y(r)}{2n}\big\{X-\sum\nolimits_{\,j}\eta^j(X)\xi_j\big\} \\
\nonumber
 & -\big((2\,n+s)\,\beta^2 + \frac{r}{2\,n} \big)\beta\big\{\,g(X,Y)\,\bar\xi + \bar\eta(X)\big( Y  - \sum\nolimits_{\,j}\eta^j(Y)\xi_j\big)
 -\sum\nolimits_{\,i}\eta^i(X)\,\eta^i(Y)\,\bar\xi\,\big\} \\
 & - 2(s-1)n\beta^3\big\{ \big(g(X,Y) -\sum\nolimits_{\,j}\eta^j(X)\,\eta^j(Y)\big)\,\bar\xi
 + \bar\eta(X)\big(Y - \sum\nolimits_{\,j}\eta^j(Y)\xi_j \big) \big\}.
\end{align}
Contracting \eqref{E-Th2-a} over $Y$ and using the well known identity ${\rm div}_g\Ric = \frac12\,dr$,
we get
\begin{align}\label{E-Th2-d}
 (n-1)\,X(r) = -\sum\nolimits_{\,i}\eta^i(X)\,\xi_i(r) - 2\,n\beta\big\{r + 2\,s\,n(2\,n + 1)\,\beta^2
 \big\}\,\bar\eta(X) .
\end{align}
Using \eqref{3.1A-f-beta} in \eqref{E-Th2-d}, gives
\begin{align}\label{E-Th2-e}
 X({r}) = - 2\,\beta\{r + 2\,s\,n(2\,n + 1)\,\beta^2\}\bar\eta(X) ,
\end{align}
hence ${r}$ is constant along the leaves of ${\cal D}$.
Applying \eqref{3.16}, \eqref{E-Th2-a} and \eqref{E-Th2-e} in \eqref{3.9},
 gives
\begin{align}\label{E-Th2-b}
 (\pounds_V R)_{X,Y}\,\xi_i = 0 .
\end{align}
Contracting \eqref{E-Th2-b} over $X$,
 we get
\begin{align}\label{E-Th2-f}
 (\pounds_V \Ric)(Y,\xi_i) = \tr\{X \to (\pounds_V R)_{X,Y}\,\xi_i\} = 0.
\end{align}
Taking the Lie derivative of $\Ric(Y,\xi_i) = -2\,n\beta^2\bar\eta(Y)$, see \eqref{2.5-f-beta}, along $V$ yields
\begin{align}\label{E-Th2-g}
 (\pounds_V \Ric)(Y,\xi_i) + \Ric(Y,\pounds_V\,\xi_i) = -2\,n\,\beta^2(\pounds_V\,\bar\eta)(Y) .
\end{align}
Using \eqref{E-Th2-f} in the preceding equation \eqref{E-Th2-g}, we obtain
\begin{align}\label{E-Th2-fg}
  \Ric(Y,\pounds_V\,\xi_i)
  = -2\,n\,\beta^2\big\{(\pounds_V\,g)(Y,\bar\xi) +g(Y,\pounds_V\,\bar\xi)\big\}
  =
  -2\,n\,\beta^2(\pounds_V\,\bar\eta)(Y)
\end{align}
for all $Y\in\mathfrak{X}_M$.
From \eqref{E-Th2-fg}, using \eqref{3.16}, we find
\begin{align}\label{E-Th2-h}
  \big\{\frac{r}{2\,n} + s\,\beta^2\big\} \big\{g(Y,\pounds_V\,\xi_i) -
 \sum\nolimits_{\,j}\eta^j(Y)\eta^j(\pounds_V\,\xi_i) \big\}
  - 2\,n\beta^2\bar\eta(Y)\bar\eta(\pounds_V\,\xi_i)
 = - 2\,n\beta^2g(Y,\pounds_V\,\bar\xi) .
\end{align}
For $Y\in{\cal D}$, \eqref{E-Th2-h} reduces to the following:
\begin{align}\label{AJ-1}
 \big(s\beta^2+\frac{r}{2\,n}\big)\,g(Y,\pounds_V\,\xi_i) = -2\,n\,\beta^2 g(Y,\pounds_V\,\bar\xi)\quad (Y\in{\cal D}),
\end{align}
from which we obtain
\begin{align}\label{E-Th2-i}
 \big\{{r} + 2\,s\,n(2\,n+1)\beta^2\big\}\,g(Y,\pounds_V\,\bar\xi)
 = 0\quad (Y\in{\cal D}).
\end{align}

\textbf{Case 1}. Let's assume that $(M,g)$ has
scalar curvature
${r}=-2\,s\,n(2\,n+1)\beta^2$, then \eqref{3.16} yields
\[
 {\rm Ric}^\sharp X = -2\,s\,n\beta^2 \big\{ X -\sum\nolimits_{\,j}\eta^j(X)\,\xi_j\big\} -2\,n\,\beta^2\,\bar\eta(X)\,\bar\xi .
\]
Hence, $(M,g)$ is an $\eta$-Einstein manifold \eqref{Eq-2.10} with $a=-2\,s\,n\beta^2$ and $b=2(s-1)n\beta^2$.

\textbf{Case 2}. Let's assume that ${r}\ne-2\,s\,n(2\,n+1)\beta^2$ on an open set ${\cal U}\subset M$.
Then $\pounds_V\,\bar\xi=0$ on $\mathcal{U}$, see \eqref{E-Th2-i} and \eqref{Eq-normal-2}.
Let us show that this leads to a contradiction.
If $\pounds_V\,\xi_i\ne0$ for some $i$, then from \eqref{AJ-1} and $\pounds_V\,\xi_i\in{\cal D}$, see \eqref{Eq-normal-2}, we get $s\beta^2+\frac{r}{2\,n}=0$.
Using the previous equality in \eqref{3.16}, we get $\Ric^\sharp X = 0$ for all $X\in{\cal D}$.
By this and \eqref{2.5-f-beta}, the following is true:
 $\Ric^\sharp X = -2\,n\,\beta^2\bar\eta(X)\,\bar\xi$.
Therefore, using Lemma~\ref{Lem-1}, we obtain
\[
 (\nabla_{\xi_i}\Ric^\sharp)X
 = -2\,n\beta^2\nabla_{\xi_i}\big(\bar\eta(X)\bar\xi\big) - \Ric^\sharp(\nabla_{\xi_i}X)
 = 0\quad (1\le i\le s).
\]
By the previous equality, $(\nabla_{\bar\xi}\Ric^\sharp)X = 0$ is true, hence by Theorem~\ref{Th-01},
we get $r=-2\,s\,n(2\,n+1)\,\beta^2$ -- a contradiction.
Therefore, $\pounds_V\,\xi_i=[V,\xi_i]=0$ for all $i$ on some open set $\mathcal{V}\subset \mathcal{U}$.
It~follows that
\begin{align}\label{AK}
 \nabla_{\xi_i}V=\nabla_V\,\xi_i = \beta\{V-\sum\nolimits_{\,j}\eta^j(V)\,\xi_j\},
\end{align}
where we have used \eqref{2.3b}.
Recall the formula, e.g.,
\cite{yano1970integral}:
\begin{align}\label{ff}
 & (\pounds_V\nabla)(X,Y) = \nabla_X\nabla_Y V - \nabla_{\nabla_X Y}V + R_{V,X} Y .
\end{align}
Replacing $Y$ by $\xi_i$ in \eqref{ff} and using \eqref{2.3b}, \eqref{2.4} and \eqref{AK}, we get
\begin{align}\label{LV-Nabla}
 (\pounds_V\nabla)(X,\xi_i) = -\beta^2\{g(X,V) - \sum\nolimits_{\,j}\eta^j(X)\eta^j(V)\}\bar\xi.
\end{align}
Further, from \eqref{3.7} and \eqref{LV-Nabla}, dividing by $2\,\beta$ we get
\begin{align*}
 \Ric^\sharp QX + 2\,s n\beta^2 QX + 2\,n\{\bar\eta(X)\bar\xi - s\sum\nolimits_j \eta^j(X)\xi_j\}
  = -(1/2)\{g(X,V) - \sum\nolimits_{\,j}\eta^j(X)\eta^j(V)\}\bar\xi .
\end{align*}
The ${\cal D}$-component of the previous equation is
\begin{align}\label{b-Ric}
 \Ric^\sharp QX + 2\,s n\beta^2 QX = 0 .
\end{align}
Since ${\cal D}$ is invariant for the non-singular operator $Q$, from \eqref{b-Ric} we get $\Ric^\sharp X = - 2\,s\,n\beta^2 X$ for all $X\in{\cal D}$.
From this and \eqref{2.5-f-beta}
we get a contradiction: $r=-2\,s\,n(2\,n+1)\beta^2$ on ${\cal V}$.
\end{proof}


\begin{definition}\rm
A vector field $X$ on a weak metric $f$-manifold $M^{2n+s}({f},Q,\xi_i,\eta^i,g)$ is called a \textit{contact vector field}
if there exists a function $\sigma\in  C^\infty(M)$ such that
\begin{align}\label{E-ict-1}
 \pounds_X\,\eta^i = \sigma\,\eta^i\quad(1\le i\le s) .
\end{align}
and if $\rho=0$, i.e., the flow of $X$ preserves $\eta^i$, then $X$ is called a \textit{strictly contact vector field}.
\end{definition}

\begin{theorem}\label{T-002a}
Let $M^{2n+s}({f},Q,\xi_i,\eta^i,g)$ be a weak $\beta f$-KM.
Suppose that $(g,V)$ represents a  $\ast$-$\eta$-RS \eqref{E-sol-ast-eta}, whose potential vector field $V$ is a contact vector field.
Then $V$ is strictly contact
and $(M, g)$ is an $\eta$-Einstein manifold \eqref{Eq-2.10} with $a=-2\,s\,n\beta^2$ and $b=2(s-1)n\beta^2$ of constant scalar curvature $r=-2\,s\,n(2\,n+1)\beta^2$.
\end{theorem}

\begin{proof}
Taking the Lie derivative of $\eta^i(X)=g(X,\xi_i)$ along $V$
and using \eqref{E-3.20} and \eqref{E-ict-1}, we obtain
 $\pounds_V\,\xi_i = (\sigma- 2(\lambda+\mu))\,\xi_i$.
From this and \eqref{E-3.21} we get $\sigma=\lambda+\mu$, and by Proposition~\ref{T-lambda0}, $V$ is strict.
Moreover,  $\pounds_V\,\xi_i = 0$. Also, \eqref{E-ict-1} yields $\pounds_V\,\eta^i=(\lambda+\mu)\,\eta^i=0$.
Recall the formula, e.g., \cite[Eq.~(5.13)]{yano1970integral}):
\begin{align}\label{3.21}
 (\pounds_{V}\nabla)(X, Y) = \pounds_{V}(\nabla_{X} Y) - \nabla_{X}\,(\pounds_{V} Y) - \nabla_{\pounds_{V} X}Y .
\end{align}
Setting $Y=\xi_i$ in \eqref{3.21} and using \eqref{2.3b}, we find
\begin{align}\label{E-ict-3}
 (\pounds_V\nabla)(X,\xi_i) = -\nabla_X(\pounds_V\,\xi_i) -\beta\sum\nolimits_{\,j}\{(\pounds_V\,\eta^j)(X)\xi_j + \eta^j(X)\pounds_V\,\xi_j\} = 0.
\end{align}
Using \eqref{E-ict-3} in \eqref{3.7} yields
 $\Ric^\sharp QX = - 2\,s\,n\beta^2 QX\ (X\in{\cal D})$.
Since ${\cal D}$ is invariant for the non-singular operator $Q$, we get $\Ric^\sharp X = - 2\,s\,n\beta^2 X$ for all $X\in{\cal D}$.
Hence $(M, g)$ is an $\eta$-Einstein manifold with $a=-2 s\,n\beta^2$, $b=2 (s-1)\,n\beta^2$, see \eqref{Eq-2.10} and \eqref{3.16},
of scalar curvature
${r}=-2\,s\,n(2\,n+1)\beta^2$.
\end{proof}


\begin{theorem}\label{T-004}
Let $(g,V)$ represents a  $\ast$-$\eta$-RS \eqref{E-sol-ast-eta} on a weak $\beta f$-KM $M^{2n+s}({f},Q,\xi_i,\eta^i,g)$.
If
$V$ is orthogonal to ${\cal D}$, i.e., $V=\sum_{\,i}\delta^{\,i}\xi_i\ne0$ for some smooth functions $\delta^{\,i}$ on $M$,
then $X(\delta^{\,i})=0$ for all $i$ and $X\in{\cal D}$; moreover, $(M, g)$ is an $\eta$-Einstein manifold \eqref{Eq-2.10}
with $a=-2\,s\,n\beta^2$ and $b=2(s-1)n\beta^2$ of constant scalar curvature $r=-2\,s\,n(2\,n+1)\beta^2$.
\end{theorem}

\begin{proof}
From \eqref{E-sol-eta2}, using \eqref{3.3C} and Proposition~\ref{T-lambda0}, we obtain
\begin{align}\label{E-T4-1}
\notag
 & g(\nabla_X V, Y) + g(X, \nabla_Y V) = - 2\Ric(X,QY) + 2\,\lambda\,g(X,Y) - 2\,s\,(2\,n-1)\beta^2 g(QX,Y) \\
 & + 2\,(s\,(2\,n-1)\beta^2-\lambda)\sum\nolimits_{\,k}\eta^k(X)\eta^k(Y) - 4\,n\beta^2\bar\eta(X)\bar\eta(Y).
 \end{align}
Using the condition $V=\sum_{\,i}\delta^{\,i}\xi_i$ and \eqref{2.3d} in \eqref{E-T4-1}, we obtain
\begin{align}\label{E-T4-2}
\notag
 & \sum\nolimits_{\,i}\{(d\delta^{\,i})(X)\,\eta^i(Y) +(d\delta^{\,i})(Y)\,\eta^i(X)\}
 +2\,\bar\delta s\,\beta\{g(X,Y) -\sum\nolimits_{\,k}\eta^k(X)\eta^k(Y)\} \\
\notag
 & = -2\Ric(X,QY) +2\,\lambda\,g(X,Y) -2\,s\,(2\,n-1)\beta^2 g(QX,Y) \\
 & + 2\,(s\,(2\,n-1)\beta^2-\lambda)\sum\nolimits_{\,k}\eta^k(X)\eta^k(Y) - 4\,n\beta^2\bar\eta(X)\,\bar\eta(Y),
\end{align}
where $\bar\delta:=\sum_{\,i}\delta^{\,i}$.
Substituting $Y=\xi_j$ in \eqref{E-T4-2} and using \eqref{2.3b} and \eqref{2.5-f-beta}, gives
\begin{align}\label{E-T4-3}
 (d\delta^{\,j})(X) +\sum\nolimits_{\,i}(d\delta^{\,i})(\xi_j)\,\eta^i(X)
 = \big\{2\,\lambda -2\,s\,(2\,n-1)\beta^2 + 2\,(s\,(2\,n-1)\beta^2-\lambda)\big\}\eta^j(X) = 0.
\end{align}
By \eqref{E-T4-3}, the functions $\delta^{\,j}\ (1\le j\le s)$ are invariant along
${\cal D}$, that is $(d\delta^{\,j})(X)=0$ for all $X\in{\cal D}$.
Hence, \eqref{E-T4-2} along ${\cal D}$ reads as
\begin{align}\label{E-T4-4b}
  \Ric(X,QY) = (\lambda - \bar\delta s\,\beta) g(X,Y) -s\,(2\,n-1)\beta^2 g(QX,Y)\quad (X,Y\in{\cal D}) .
\end{align}
Taking the covariant derivative of \eqref{E-T4-4b} in $Z$ and using \eqref{E-nS-10b} yields
\begin{align}\label{E-T4-4c}
\notag
 (\nabla_Z\Ric)(X,QY) & = \beta\,\bar\eta(Y)\,\Ric(X, \widetilde Q Z )
 - s\,(2\,n-1)\beta^3\bar\eta(Y)\,g(\widetilde Q Z, X)\\
 & + s\,(4\,n-1)\beta^3\bar\eta(X)\,g(\widetilde Q Z, Y)\quad (X,Y\in{\cal D}) .
\end{align}
Using $Z=\xi_i$ in \eqref{E-T4-4c}, we get $(\nabla_{\xi_i}\Ric)(X,QY)=0\ (X,Y\in{\cal D})$.
Applying to this \eqref{3.1},
gives
 $g(\Ric^\sharp X + 2\,s n\,\beta^2 X, QY) = 0\ (X,Y\in{\cal D})$.
From the above, in view of non-singularity of $Q$, we find
\begin{align}\label{E-T4-4A}
 \Ric^\sharp X + 2\,s n\,\beta^2 X = \sum\nolimits_{\,k}\alpha^k(X)\xi_k
\end{align}
for all $X\in\mathfrak{X}_M$ and some functions $\alpha^k(X)$ on $M$.
Taking the inner product of \eqref{E-T4-4A} with $\xi_i$ and using \eqref{2.5-f-beta} gives $\alpha^j(X) = 2\,s\,n\beta^2\eta^j(X)-2\,n\beta^2\bar\eta(X)$.
Hence
\begin{align}\label{E-T4-4D}
 \Ric^\sharp X = -2\,s\,n\,\beta^2\big\{ X - \sum\nolimits_{\,j}\eta^j(X)\,\xi_j\big\} - 2\,n\,\beta^2\,\bar\eta(X)\,\bar\xi.
\end{align}
By \eqref{E-T4-4D} we conclude that
$g$ is an $\eta$-Einstein metric \eqref{Eq-2.10} with $a=-2\,s\, n\beta^2$ and $b=2(s-1)n\beta^2$.
\end{proof}

\begin{definition}\rm
If $v,\lambda,\mu\in C^\infty(M)$ and $V=\nabla v$, then \eqref{E-sol-ast-eta} defines a \textit{gradient almost $\ast$-$\eta$-RS}:
\begin{align}\label{E-sol-ast-eta-grad}
 {\rm Hess}_{\,v} +\Ric^\ast = \lambda\,\big\{ g -\sum\nolimits_{\,i}\eta^i\otimes\eta^i\big\} +(\lambda+\mu)\bar\eta\otimes\bar\eta
 \quad
 \mbox{(see \cite{DeyRoy-2020,DSB-2024,ven-2022} for $s=1$)}.
\end{align}
The Hessian ${\rm Hess}_{\,v}$ is a symmetric (0,2)-tensor defined by
 ${\rm Hess}_{\,v}(X,Y) = g(\nabla_X\nabla v, Y)$ for $X,Y\in\mathfrak{X}_M$.
\end{definition}

For a gradient almost $\ast$-$\eta$-RS, \eqref{E-sol-eta2} reduces to the following:
\begin{align}\label{E-sol-eta2-grad1}
\notag
 & {\rm Hess}_{\,v}(X,Y) + g(\Ric QX, Y)
 = \lambda\,g(X,Y) - s\,(2\,n-1)\beta^2 g(QX,Y) \\
 & +(s\,(2\,n-1)\beta^2-\lambda)\sum\nolimits_{\,j}\eta^j(X)\eta^j(Y) + (\lambda+\mu-2\,n\beta^2)\bar\eta(X)\bar\eta(Y) ,
\end{align}
where $v,\lambda,\mu\in C^\infty(M)$ and $Q$ commutes with the Ricci operator $\Ric^\sharp$.
We can write \eqref{E-sol-eta2-grad1} as
\begin{align}\label{E-sol-eta2-grad}
\notag
 & \nabla_X\nabla v + Q\Ric^\sharp X = \lambda\,X - s\,(2\,n-1)\beta^2 QX \\
 & +(s\,(2\,n-1)\beta^2-\lambda)\sum\nolimits_{\,j}\eta^j(X)\xi_j + (\lambda+\mu-2\,n\beta^2)\bar\eta(X)\,\bar\xi .
\end{align}

\begin{theorem}\label{T-005}
Let a weak $\beta f$-KM $M^{2n+s}({f},Q,\xi_i,\eta^i,g)$ be a gradient almost $\ast$-$\eta$-RS~\eqref{E-sol-ast-eta-grad}.
Then
$\nabla v=\sum_{\,i}\delta^{\,i}\xi_i\ne0$ for some smooth functions $\delta^{\,i}=const$ along ${\cal D}$ for all $i$ on $M$,
and $(M, g)$ is an $\eta$-Einstein manifold \eqref{Eq-2.10}
with $a=-2\,s\,n\beta^2$ and $b=2(s-1)n\beta^2$ of scalar curvature $r=-2\,s\,n(2\,n+1)\beta^2$.
\end{theorem}

\begin{proof}
Differentiating
\eqref{E-sol-eta2-grad} along
$Y$ and using \eqref{2.3b}, \eqref{2.3c}, \eqref{E-nS-10b} and Proposition~\ref{T-lambda0}, yields
\begin{align}\label{E-sol-eta2-grad2}
\notag
 & \nabla_Y\nabla_X\nabla v + (\nabla_Y Q\Ric^\sharp)X = Y(\lambda)\big\{ X - \sum\nolimits_{\,k}\eta^k(X)\xi_k\big\}
 - s(2\,n-1)\beta^2 (\nabla_Y Q)X \\
\notag
 & +((2\,n-1)\beta^2-\lambda)\sum\nolimits_{\,j}\{(\nabla_Y\eta^j)(X)\xi_j + \eta^j(X)\nabla_Y\xi_j\}
 -2\,n\beta^2\{ (\nabla_Y\bar\eta)(X)\bar\xi + \bar\eta(X)\nabla_Y\bar\xi \} \\
\notag
 & = Y(\lambda)\big\{ X - \sum\nolimits_{\,k}\eta^k(X)\xi_k\big\} + (2\,n-1)\beta^3\big\{\bar\eta(X)\,\widetilde Q Y - g(\widetilde Q Y, X)\,\bar\xi\big\} \\
 & +\beta\big[(2\,(1-s)n -1)\beta^2-\lambda \big]\big\{\{g(X,Y) -\sum\nolimits_{\,k}\eta^k(X)\,\eta^k(Y)\}\bar\xi
 + \bar\eta(X)\{Y -\sum\nolimits_{\,k}\eta^k(Y)\,\xi_k\}\big\} .
\end{align}
Applying \eqref{E-sol-eta2-grad2} and \eqref{E-sol-eta2-grad} in the expression of Riemannian curvature tensor, we obtain
\begin{align}\label{E-sol-eta2-R}
\notag
 & R_{X,Y}\nabla v =  (\nabla_Y Q\Ric^\sharp)X -(\nabla_X Q\Ric^\sharp)Y \\
\notag
 & + X(\lambda)\big\{ Y - \sum\nolimits_{\,k}\eta^k(Y)\xi_k\big\} - Y(\lambda)\big\{ X - \sum\nolimits_{\,k}\eta^k(X)\xi_k\big\}
 + s(2\,n-1)\beta^3\big\{\bar\eta(Y)\,\widetilde Q X - \bar\eta(X)\,\widetilde Q Y  \big\} \\
 & +\beta\big[(2\,(1-s)n -1)\beta^2-\lambda \big]\big\{\bar\eta(Y)\{X -\sum\nolimits_{\,k}\eta^k(X)\,\xi_k\}
 - \bar\eta(X)\{Y -\sum\nolimits_{\,k}\eta^k(Y)\,\xi_k\}\big\} .
\end{align}
An inner product of \eqref{E-sol-eta2-R} with $\xi_i$ and use of \eqref{E-nS-10b}, \eqref{2.5-f-beta}, \eqref{E-L-02a} and Proposition~\ref{T-lambda0} yields
\begin{align}\label{E-sol-eta2-R-xi}
g(R_{X,Y}\nabla v, \xi_i)
 = 0 .
\end{align}
On the other hand, using \eqref{2.4} we obtain
\begin{align}\label{E-sol-eta2-R-xi2}
\notag
 & g(R_{X,Y}\nabla v, \xi_i) = - g(R_{X,Y}\xi_i, \nabla v) \\
 & = \beta^2\bar\eta(Y) \{X(v) - \sum\nolimits_{\,j}\eta^j(X)\xi_j(v)\} -\beta^2\bar\eta(X)\{ Y(v) - \sum\nolimits_{\,j}\eta^j(Y)\xi_j(v) \} .
\end{align}
Comparing \eqref{E-sol-eta2-R-xi} and \eqref{E-sol-eta2-R-xi2}
and using the assumption $\beta\ne0$, we conclude that
\begin{align}\label{E-sol-eta2-R-xi3}
 \bar\eta(Y) \{X(v) - \sum\nolimits_{\,j}\eta^j(X)\xi_j(v)\} = \bar\eta(X)\{ Y(v) - \sum\nolimits_{\,j}\eta^j(Y)\xi_j(v) \}.
\end{align}
Taking $Y=\xi_i$ in \eqref{E-sol-eta2-R-xi3}, we obtain
 $X(v) = \sum\nolimits_{\,j}\eta^j(X)\xi_j(v)$.
Thus, $X(v)=0$ for all $X\in {\cal D}$, i.e., the soliton vector field $V=\nabla v$ is orthogonal to ${\cal D}$.
Applying Theorem~\ref{T-004} completes the proof.
\end{proof}

\section{Conclusion}

The concept of the $\ast$-Ricci tensor of S.\,Tashibana was adapted to weak metric $f$-structure introduced by the author and R. Wolak.
The $\ast$-$\eta$-RS were defined and their interaction with weak $\beta f$-KM -- a distinguished class of weak metric $f$-manifolds which
for $s>1$ are not Einstein~mani\-folds -- was studied.
Conditions under which a weak $\beta f$-KM equipped with a $\ast$-$\eta$-RS structure carries $\eta$-Einstein metric were found.
A future task may be to study in detail the interaction of almost $\ast$-$\eta$-RS (i.e., $\lambda,\mu$ in \eqref{E-sol-ast-eta} are smooth functions)
with weak $\beta f$-KM as well as other weak metric~$f$-manifolds.



\end{document}